\documentclass[11pt]{amsart}
\usepackage{setspace}
\usepackage{amsmath, amsfonts, amsthm, amstext, xspace}
\usepackage{amssymb,latexsym}
\usepackage{comment}
\usepackage{enumitem}
\usepackage{xcolor}
\usepackage{hyperref}
\definecolor{seagreen}{RGB}{46,139,87}
\definecolor{maroon}{RGB}{128,0,0}
\definecolor{darkviolet}{RGB}{148,0,211}
\definecolor{twelve}{RGB}{100,100,170}
\definecolor{thirteen}{RGB}{100,150,50}
\definecolor{fourteen}{RGB}{200,0,0}
\definecolor{fifteen}{RGB}{0,200,0}
\definecolor{sixteen}{RGB}{0,0,200}
\definecolor{seventeen}{RGB}{200,0,200}
\definecolor{eighteen}{RGB}{0,200,200}

\usepackage[all]{xy}
\newtheorem{thm}{Theorem}[section]
\newtheorem*{theorem*}{Theorem}
\newtheorem*{conjecture*}{Conjecture}
\newtheorem*{corollary*}{Corollary}
\newtheorem{lem}[thm]{Lemma}
\newtheorem{cor}[thm]{Corollary}

\newtheorem{prop}[thm]{Proposition}

\newtheorem{conv}[thm]{Convention}

\theoremstyle{definition}
\newtheorem{defin}[thm]{Definition}
\newtheorem{exm}[thm]{Example}

\newtheorem{rem2}[thm]{Remark}

\newtheorem{thmx}{Theorem}

\def\c{\mathbb{C}}
\def\e{\mathbb{E}}

\def\k{\mathbb{K}}

\def\m{\mathbb{M}}

\def\p{\mathbb{P}}
\def\q{\mathbb{Q}}
\def\r{\mathbb{R}}

\def\z{\mathbb{Z}}

\def\cf{\mathcal{F}}

\def\holim{\operatorname{lim}}

\def\hocolim{\operatorname{colim}}

\def\colim{\operatorname{colim}}
\def\ker{\operatorname{ker}}
\def\coker{\operatorname{coker}}
\def\id{\operatorname{id}}

\newcommand{\limt}[1]{\underset{#1}{\lim \,}}
\def\vbar{\overline{v}}

\usepackage{amssymb}
\usepackage{tikz-cd}

\author{Guchuan Li}\address{University of Michigan}\email{guchuan@umich.edu}
\author{Vitaly Lorman}\address{Children's Hospital of Philadelphia}\email{lormanv@chop.edu}
\author{J.D. Quigley}\address{Cornell University}\email{jdq27@cornell.edu}
\title{Tate blueshift and vanishing for Real oriented cohomology}

\begin{document}
\maketitle

\begin{abstract}
We study transchromatic phenomena for the Tate construction of Real oriented cohomology theories. First, we show that after suitable completion, the Tate construction with respect to a trivial $\z/2$-action on height $n$ Real Johnson--Wilson theory splits into a wedge of height $n-1$ Real Johnson--Wilson theories. This is the first example of Tate blueshift at all chromatic heights outside of the complex oriented setting. Second, we prove that the Tate construction with respect to a trivial finite group action on Real Morava K-theory vanishes, refining a classical Tate vanishing result of Greenlees--Sadofsky. In the course of proving these results, we develop some ideas in equivariant chromatic homotopy theory (e.g., completions of module spectra over Real cobordism, $C_2$-equivariant chromatic Bousfield localizations) and apply the parametrized Tate construction. 
\end{abstract}

\tableofcontents

\section{Introduction} 

\subsection{History, motivation, and main results}

Chromatic homotopy theory provides a filtration of the stable homotopy category of finite $p$-local complexes. If $X$ is a finite $p$-local complex, then the Hopkins--Smith thick subcategory theorem \cite{HS98} and Hopkins--Ravenel chromatic convergence theorem \cite{Rav92} provide a filtration of $X$ whose $n$-th level is given by the Bousfield localization $L_{E(n)}X$ of $X$ with respect to the $n$-th Johnson--Wilson theory $E(n)$. On homotopy groups, these localizations decompose $\pi_*(X)$ into periodic families of elements of increasing periodicity \cite{DHS88}. However, even for the simplest finite complexes like the sphere, very little is known about $\pi_*L_{E(n)}X$ for $n \geq 3$, cf. the computations of Miller--Ravenel--Wilson \cite{MRW77}. Because of the difficulties in analyzing the stable homotopy category beyond filtration 2, it is desirable to find ways of comparing information from different levels of the filtration; this is sometimes called \emph{transchromatic homotopy theory}. 

The purpose of this paper is to study transchromatic phenomena for the Tate construction of Real oriented cohomology theories. Let $E$ be a cohomology theory equipped with an action by a finite group $\Sigma$. The \emph{$\Sigma$-Tate construction} of $E$, denoted $E^{t\Sigma}$, is the cofiber of the additive norm map from the $\Sigma$-homotopy orbits $E_{h\Sigma}$ to the $\Sigma$-homotopy fixed points $E^{h\Sigma}$ \cite{GM95}. We discuss Real oriented cohomology theories further in Section \ref{SSS:RealOr}.

Over the past four decades, the Tate construction has been observed to produce interesting transchromatic phenomena when applied to chromatic cohomology theories equipped with a trivial $\Sigma$-action. We mention several examples here\footnote{Our list of examples expands on the list appearing in the introduction to \cite{BR19}.}:
\begin{enumerate}
	\item[\textbf{1984:}] In \cite{DM84}, Davis--Mahowald showed $BP\langle 1 \rangle^{t\Sigma_2}$ and $ko^{t\Sigma_2}$ split into wedges of suspensions of $H\z_2$. Note that $BP\langle 1 \rangle$ and $ko$ have chromatic height one and $H\z_2$ has chromatic height zero. This sort of decrease in chromatic height after applying the Tate construction is referred to as \emph{Tate blueshift}. 
	\item[\textbf{1986:}] In \cite{DJKMW86}, Davis--Johnson--Klippenstein--Mahowald--Wegmann showed  $BP\langle 2 \rangle^{t\Sigma_2}$ splits into a wedge of suspensions of $BP\langle 1 \rangle$ after $p$-completion. This was the first example of Tate blueshift from height two to height one; it motivated the conjecture that similar blueshift results hold at all chromatic heights. 
	\item[\textbf{1995:}] In \cite{GM95}, Greenlees--May proved the first non-connective Tate blueshift result, showing that $KU^{t\Sigma_2}$ and $KO^{t\Sigma_2}$ split into wedges of suspensions of $H\q$. 
	\item[\textbf{1996:}] In \cite{GS96}, Greenlees--Sadofsky observed the first instance of \emph{Tate vanishing} by showing that $K(n)^{t\Sigma} \simeq *$ for any finite group $\Sigma$. 
	\item[\textbf{1998:}] In \cite{AMS98}, Ando--Morava--Sadofsky proved the first instance of Tate blueshift at all chromatic heights: after $K(n-1)$-localization, $E(n)^{t\Sigma_2}$ splits into a wedge of suspensions of $E(n-1)$. 
	\item[\textbf{2019:}] In \cite{BR19}, Bailey--Ricka showed that $\mathit{tmf}^{t\Sigma_2}$ splits into a wedge of suspensions of $ko$ after $2$-completion. Note that $\mathit{tmf}$ has chromatic height two and $ko$ has chromatic height one, and neither cohomology theory is complex orientable. This is the first example of Tate blueshift where neither cohomology theory is complex orientable.
\end{enumerate}
In particular, Tate blueshift and vanishing have been observed at all chromatic heights in the complex oriented setting, but only at heights one and two outside of the complex oriented setting (where computations tend to be far more difficult and much less is known). Our main theorems provide the first examples of Tate blueshift and vanishing at all chromatic heights for cohomology theories which are not complex orientable:

\begin{thmx}[Real oriented Tate blueshift, Theorem \ref{ERtatesplitting}]\label{ThmA} Let $\lambda_n := 2^{n+2}(2^{n-1}-1)+1$ and let $\Sigma_2$ act trivially on the $n$-th \emph{Real Johnson--Wilson theory} $ER(n)$ (see Section \ref{Sec:RJW} for a definition).
\begin{enumerate}
\item
 For $n \geq 2$ there is a map of spectra
$$\limt{i} \bigvee_{j \leq i} \Sigma^{(1-\lambda_{n-1}) j} ER(n-1) \longrightarrow (ER(n)^{t\Sigma_2})^\wedge_{\widehat{I}_{n-1}}$$
that becomes an isomorphism on homotopy groups after completion at $\widehat{I}_{n-1}$, or equivalently, after $K(n-1)$-localization. Here, $\widehat{I}_{n-1}$ is the ideal $\bar{I}_{n-1} = (p,\bar{v}_1,\ldots,\bar{v}_{n-2})$ in $\e(n)_\star$ shifted into integer degrees using the invertible class $y \in \pi_{\lambda+\sigma}\e(n)$. (See Section \ref{sec:coefficients} for more on $ER(n)$, $\e(n)$, and their coefficients.)
\item
For $n=1$, there is an equivalence of spectra
$$\limt{i} \bigvee_{j \leq i} \Sigma^{4j} ER(0) \simeq ER(1)^{t\Sigma_2},$$
where $ER(0)\simeq H\mathbb{Q}$ and $ER(1) \simeq KO_{(2)}$.
\end{enumerate}
\end{thmx}

\begin{thmx}[Real oriented Tate vanishing, Theorem \ref{Thm:realvanishing}]\label{ThmB}
Let $\Sigma$ be a finite group and let $\Sigma$ act trivially on the \emph{$n$-th Real Morava K-theory} $KR(n)$ (see Section \ref{Sec:RJW} for a definition). Then 
$$KR(n)^{t\Sigma} \simeq *.$$
\end{thmx}

Before discussing these results in depth, we mention three applications:

\begin{rem2}
The $n=1$ case of Theorem \ref{ThmA} recovers Greenlees and May's identification of $KO^{t\Sigma_2}$ mentioned above. Unlike their proof, which relied on the Atiyah--Segal Completion Theorem, our proof is purely homotopical. 
\end{rem2}

\begin{rem2}
Recent work of the first author with Shi, Wang, and Xu \cite{LSWX19} suggests that the Hurewicz image of $ER(n)$ becomes increasingly large as $n$ increases.\footnote{For example, $ER(n)$ detects the first $n$ Hopf invariant one and Kervaire invariant one elements and the first $n-1$ $\bar{\kappa}$-elements in the stable homotopy groups of spheres, if they are nontrivial.} Theorem \ref{ThmA} therefore provides a means of analyzing the complicated Hurewicz image of $ER(n)$ using the simpler Hurewicz image of $ER(n-1)$. 
\end{rem2}

\begin{rem2}
The proof of \cite[Cor. 1.2]{GS96} holds with $K=KR(n)$, so Theorem \ref{ThmB} implies that the duality map induced by the transfer
$$KR(n) \wedge B\Sigma_+ \to F(B\Sigma_+, KR(n))$$
is an equivalence, i.e., $B\Sigma_+$ is self-dual with respect to $KR(n)$. 
\end{rem2}

\subsection{Overview of the main ideas}

There are two key ideas going into the proofs of Theorems \ref{ThmA} and \ref{ThmB}. The first is the use of \emph{Real oriented cohomology theories}, which are genuine $C_2$-equivariant refinements of complex oriented cohomology theories. The second is the parametrized Tate construction, a genuine $C_2$-equivariant refinement of the ordinary Tate construction, which we discuss further in Section \ref{SS:Gen}.

\subsubsection{Real oriented cohomology theories}\label{SSS:RealOr}

Real oriented cohomology theories are genuine $C_2$-equivariant cohomology theories equipped with a choice of Thom class for Real vector bundles. The primary examples are the K-theory of Real vector bundles $K\r$ \cite{Ati66}, Real cobordism $\m\r$ \cite{Lan68}, certain forms of topological modular forms with level structure \cite{HM17}, and the two examples most relevant to this paper, Real Johnson--Wilson theory $\e(n)$ and Real Morava K-theory $\k(n)$ \cite{HK01}. The cohomology theories $ER(n)$ and $KR(n)$ appearing above are the categorical fixed points of $\e(n)$ and $\k(n)$, respectively. 

As their names suggest, $\e(n)$ and $\k(n)$ are genuine $C_2$-equivariant lifts of classical Johnson--Wilson theory $E(n)$ and classical Morava K-theory $K(n)$. We mentioned Tate blueshift and vanishing results for these cohomology theories already:
\begin{itemize}

\item Ando--Morava--Sadofsky \cite{AMS98} proved that there is a map of spectra
\begin{equation}\label{AMSEqn}
\limt{i} \bigvee_{j \leq i} \Sigma^{2j} E(n-1) \to (E(n)^{t\Sigma_2})^{\wedge}_{I_{n-1}}
\end{equation}
which induces an isomorphism on homotopy groups after completion at $I_{n-1}$, or equivalently after Bousfield localization with respect to $K(n-1)$.

\item Greenlees--Sadofsky \cite{GS96} showed that 
\begin{equation}\label{GSEqn}
K(n)^{t\Sigma} \simeq *,
\end{equation}
where $\Sigma$ is any finite group acting trivially on $K(n)$. 

\end{itemize}

\subsubsection{Sketch of Proof of Theorem \ref{ThmB}}

It turns out that Theorem \ref{ThmB} follows fairly directly from \eqref{GSEqn}. Kitchloo--Wilson \cite{KW07} have produced a cofiber sequence
$$\Sigma^\lambda KR(n) \overset{y}{\to} KR(n) \to K(n)$$
relating $KR(n)$ to $K(n)$. We prove Theorem \ref{ThmB} by using this cofiber sequence to bootstrap the Greenlees--Sadofsky Tate vanishing result for $K(n)$ \eqref{GSEqn} to $KR(n)$. This requires a careful analysis of the Real Morava K-theory of inverse limits.

\subsubsection{The parametrized Tate construction}\label{SS:Gen}

Surprisingly, Theorem \ref{ThmA} does \emph{not} follow easily from \eqref{AMSEqn} in the same way that Theorem \ref{ThmB} follows from \eqref{GSEqn}. Moreover, since the techniques of \cite{AMS98} rely heavily on the fact that $E(n)$ is complex oriented (e.g., to use formal group law computations in their arguments), the arguments of Ando--Morava--Sadofsky cannot be applied \emph{mutatis mutandis} with $ER(n)$ in place of $E(n)$. Therefore we need genuinely new ideas to prove Theorem \ref{ThmA}.

We prove Theorem \ref{ThmA} by producing a genuine $C_2$-equivariant refinement of the Ando--Morava--Sadofsky equivalence \eqref{AMSEqn} and then passing to categorical fixed points. The key to producing this refinement is the \emph{parametrized Tate construction}, denoted $(-)^{t_{C_2}\Sigma_2}$, which is a genuine $C_2$-equivariant refinement of the ordinary Tate construction $(-)^{t\Sigma_2}$ \cite{QS21a}.\footnote{This is the generalized Tate construction \cite{GM95} for the family $\cf$ of subgroups $H \subseteq \Sigma_2 \rtimes C_2$ with $H \cap \Sigma_2 = 1$.} The parametrized Tate construction is obtained by replacing the universal space $E\Sigma_2$ by the $C_2$-equivariant universal space $E_{C_2}\Sigma_2$ in the definition of the Tate construction; we discuss the idea behind this replacement further in Section \ref{SSS:Replacements}.

We now state our main equivariant theorem.

\begin{thmx}[Parametrized blueshift, Theorem \ref{C2eqvtERsplit}]\label{ThmC}
Let $\Sigma_2$ act trivially on $\e(n)$. There is a map of genuine $C_2$-spectra
$$\limt{i} \bigvee_{j \leq i} \Sigma^{\rho j} \e(n-1) \to (\e(n)^{t_{C_2}\Sigma_2})^{\wedge}_{{I}_{n-1}}$$
which induces an isomorphism on $C_2$-equivariant homotopy groups after completion at $\overline{I}_{n-1} = (p,\bar{v}_1,\ldots,\bar{v}_{n-2})$ or equivalently after $\k(n-1)$-localization. Here, $\rho$ denotes the regular representation of $C_2$. 
\end{thmx}

We defer an in-depth discussion of Theorem \ref{ThmC} and the parametrized Tate construction until Section \ref{SS:ThmC}. However, we note that the parametrized Tate construction also fits into a genuine $C_2$-equivariant refinement of Theorem \ref{ThmB}:

\begin{thmx}[Parametrized vanishing, Theorem \ref{Thm:TVMain}]\label{ThmRV}
Let $G$ be a finite abelian group and let $G$ act trivially on $\k(n)$. Then 
$$\k(n)^{t_{C_2}G} \simeq *.$$
\end{thmx}

\begin{rem2}
We emphasize that Theorem \ref{ThmRV} is \emph{not} necessary for the proof of Theorem \ref{ThmB} since we can use the Kitchloo--Wilson cofiber sequence to obtain a more direct proof. However, Theorem \ref{ThmRV} does have a concrete application in $C_2$-equivariant homotopy theory. By replacing the components in the proof of \cite[Cor. 1.2]{GS96} by their parametrized counterparts, we find that the duality map induced by the transfer
$$\k(n) \wedge B_{C_2}\Sigma_+ \to F(B_{C_2}\Sigma_+, \k(n))$$
is an equivalence of $C_2$-spectra, i.e., the $C_2$-equivariant classifying space $B_{C_2}\Sigma_+$ is self-dual with respect to $\k(n)$. 
\end{rem2}

\subsubsection{Sketch of Proof of Theorem \ref{ThmA}}

The proof of Theorem \ref{ThmA} from Theorem \ref{ThmC} is not immediate: the $C_2$-fixed points of the parametrized Tate construction are not generally the ordinary Tate construction of the $C_2$-fixed points. In general, there is not even a comparison map between these two constructions; however, we show that using certain homotopy limit models for the Tate constructions (e.g., Corollary \ref{Cor:TLim}) and cofreeness (Definition \ref{Def:FreeCofree}) of $\e(n)$, one can construct a comparison map for $\e(n)$. 

\begin{thmx}[Tate comparison, Theorem \ref{Thm:Agree}]\label{ThmD}
The map
$$\e(n)^{t\Sigma_2} \to \e(n)^{t_{C_2}\Sigma_2}$$
is a $C_2$-equivariant weak equivalence. 
\end{thmx}

We prove this comparison theorem by analyzing the $RO(C_2)$-graded homotopy groups of both sides. While the homotopy of the right-hand side follows from a fairly standard formal group law argument (albeit adapted to the Real oriented setting), the homotopy of the left-hand side is more subtle. In particular, we need to employ $\e(n)$-orientability results from \cite{KW15, KLW18}, properties of the coefficients of $ER(n)$, and the $ER(n)$-cohomology of stunted projective spectra. After deducing that the homotopy groups of both sides are abstractly isomorphic, we use a cofinality argument to show that the two sides, viewed as certain inverse limits, are equivalent. 

\begin{rem2}
Some results in this work apply to general Real oriented cohomology theories, such as Theorem \ref{ThmE} below. However, the combination of ingredients needed to prove Theorem \ref{ThmD} mentioned above (cofreeness, orientability results, and understanding of the coefficients) is at present known specifically only for $\e(n)$. Proving similar results for other equivariant cohomology theories and their fixed points could be a first step towards expanding our understanding our blueshift to more spectra. 
\end{rem2}

\subsection{Further discussion of Theorem \ref{ThmC}}\label{SS:ThmC}

To explain the proof of Theorem \ref{ThmC}, we first recall in Section \ref{SSS:AMSE} the proof of the analogous classical result of Ando--Morava--Sadofsky. Each step in their proof requires a novel construction or technique in equivariant homotopy theory; we summarize these changes in Section \ref{SSS:Replacements}.

\subsubsection{The Ando--Morava--Sadofsky equivalence}\label{SSS:AMSE}

If $E$ is a complex oriented cohomology theory equipped with a trivial $\Sigma_2$-action, then there is an isomorphism
\begin{equation}\label{Eqn3}
\pi_{-*}(E^{t\Sigma_2}) \cong E^*((x))/[2](x)
\end{equation}
with $|x| =2$, where $[2](x)$ is the $2$-series of the formal group law associated to $E$. By examining $[2](x)$ for $E(n)$, one produces an isomorphism
\begin{equation}\label{Eqn4}
\pi_*(E(n)^{t\Sigma_2})_{I_{n-1}}^\wedge \cong E(n-1)_*((x))_{I_{n-1}}^\wedge.
\end{equation}
Next, one must produce the map \eqref{AMSEqn}. The map is constructed using the formula for trivial $\Sigma_2$-Tate constructions \cite{GM95}
\begin{equation}\label{GMEqn}
X^{t\Sigma_2} \simeq \lim_i (\Sigma X \wedge P^\infty_{-i}),
\end{equation}
where $P^\infty_{-i} = Th(-i\gamma \to \r \p^\infty)$. The map \eqref{AMSEqn} is obtained as the inverse limit of maps from bounded below wedges to bounded below stunted projective spectra produced using special properties of $MU$-module spectra. Finally, one identifies $I_{n-1}$-completion with $K(n-1)$-localization using the equivalence
\begin{equation}\label{MUEqn}
L_{K(j-1)}M \simeq (v_{j-1}^{-1}M)_{I_{j-1}}^\wedge
\end{equation}
which holds for any $MU$-module spectrum $M$. 

\subsubsection{Sketch of Proof of Theorem \ref{ThmC}}\label{SSS:Replacements}

The technical core of this work is concerned with proving Real oriented analogs of \eqref{Eqn3}-\eqref{MUEqn}. Our first key observation is that an analog of \eqref{Eqn3} holds for Real oriented cohomology theories if one replaces the (classical) Tate construction by its $C_2$-equivariant enrichment, the parametrized Tate construction.

\begin{thmx}[Theorem \ref{Thm:ParamTateCoeffs}]\label{ThmE}
Let $\e$ be a Real oriented cohomology theory. Then there is an $RO(C_2)$-graded isomorphism
\begin{equation}\tag{3'}\label{Eqn32}
\pi_{-\star}\e^{t_{C_2}\Sigma_2} \cong \e^\star((\bar{x}))/([2](\bar{x}))
\end{equation}
where $|\bar{x}| = \rho$. Here, $\rho$ is the regular representation of $C_2$. 
\end{thmx}

\begin{rem2}
As we have mentioned above, the parametrized Tate construction is obtained by replacing the universal space $E\Sigma_2$ by the $C_2$-equivariant universal space $E_{C_2}\Sigma_2$. With this in mind, readers familiar with Real oriented homotopy theory may not be surprised by Theorem \ref{ThmE}; $C_2$-equivariant universal spaces already appear in conjunction with Real oriented cohomology theories in the work of Hu--Kriz \cite{HK01} and Kitchloo--Wilson \cite{KW07}. Indeed, Theorem \ref{ThmE} follows from a result of Kitchloo--Wilson which can be stated in our language using  \emph{parametrized homotopy fixed points}. 
\end{rem2}

With this formula available, we adapt the formal group law computations of \cite{AMS98} to obtain an $RO(C_2)$-graded isomorphism
\begin{equation}\tag{4'}\label{Eqn42}
\pi_\star(\e(n)^{t_{C_2}\Sigma_2})_{\bar{I}_{n-1}}^\wedge \cong \e(n-1)_\star((x))_{\bar{I}_{n-1}}^\wedge.
\end{equation}
These computations are fairly similar to their classical counterparts; the key point is that we need to work with the parametrized Tate construction instead of the ordinary Tate construction in order to bring formal group laws into the picture.

To construct a map of $C_2$-spectra, we use a version of \eqref{GMEqn} for the parametrized Tate construction from work Shah and the third author \cite{QS21a}, which gives an equivalence of genuine $C_2$-spectra
\begin{equation}\tag{5'}\label{GMEqn2}
X^{t_{C_2}\Sigma_2} \simeq \lim_i (\Sigma X \wedge Q^\infty_{-i})
\end{equation}
where $Q^\infty_{-i} = Th(-i\gamma \to B_{C_2}\Sigma_2)$ is the Thom spectrum of the $C_2$-equivariant classifying space of $\Sigma_2$. 

We obtain the map in Theorem \ref{ThmC} using the formula \eqref{GMEqn2} along with several new results on completions of $\m\r$-modules which refine classical results of Greenlees--May for $MU$-modules \cite{GM97}. These results may be of independent interest.  

Finally, we note that the last statement of Theorem \ref{ThmC} relies on the following identification of certain equivariant Bousfield localizations and completions. 

\begin{thmx}[Localization and completion comparison, Theorem \ref{prop:localizationcomparison}]\label{ThmF} Let $\e$ be any $\m\r(n) := \bar{v}_n^{-1}\m\r_{(2)}$-module spectrum such that the underlying canonical map $E \to E_{I_m}^\wedge$ factors through $L_{K(m)}E$ and gives an equivalence $L_{K(m)}E \simeq E_{I_m}^\wedge$. 
\begin{enumerate}[label=(a)]
\item The canonical equivariant map $\e \to \e_{\bar{I}_m}^\wedge$ factors through $L_{\k(m)}\e$ and gives an equivalence 
\begin{equation}\tag{6'}\label{MUEqn2}
L_{\k(m)}\e\simeq \e_{\bar{I}_m}^\wedge.
\end{equation}
\item[(b)] The canonical map $ER \to ER_{\widehat{I}_m}^\wedge$ factors through $L_{K(m)}ER$ and gives an equivalence $L_{K(m)}ER \simeq ER_{\widehat{I}_m}^\wedge$ (where $ER$ denotes the categorical fixed points of $\e$).
\end{enumerate}
\end{thmx}

\subsection{Future work}
We conclude with some directions for future work.

\subsubsection{Redshift for Mahowald invariants}
In \cite{MR93}, Mahowald and Ravenel applied the Davis--Mahowald Tate splitting for $ko^{t\Sigma_2}$ to calculate the Mahowald invariants $M(2^i)$, $i \geq 1$. In particular, they showed that $M(2^i)$ is $v_1$-periodic for all $i \geq 1$. Work of the third author \cite{Qui19d} applies the Bailey--Ricka splitting for $tmf^{t\Sigma_2}$ towards calculating the iterated Mahowald invariants $M(M(2^i))$, and preliminary work along with the low-dimensional calculations of Behrens \cite{Beh07} suggest that $M(M(2^i))$ is $v_2$-periodic. 

The Mahowald invariant has been conjectured by Mahowald and Ravenel to take $v_n$-periodic classes to $v_n$-torsion classes (with some exceptions) \cite{MR84}, and empirical evidence suggests that this redshift is closely intertwined with blueshift for the Tate construction. 
The fixed point spectra $ER(n)$ have been shown to detect interesting elements of height $n$ by the work of the first author with Shi, Wang, and Xu \cite{LSWX19}. It would be interesting to apply the Tate splitting for $ER(n)^{t\Sigma_2}$ above in order to calculate Mahowald invariants of $v_n$-periodic elements; a natural starting point would be $M(\bar{\kappa})$ where $\bar{\kappa}$ is the ($v_2$-periodic) generator of $\pi_{20}(S^0)$ which is first detected in the Hurewicz image of $ER(2)$. 

\subsubsection{Tate blueshift for fixed points of Lubin-Tate theory}
Ando, Morava, and Sadofsky remark in \cite{AMS98} that after appropriate completion, analogs of their blueshift results hold for Lubin--Tate theory. Since Lubin--Tate theory is Real oriented \cite{HS17}, we expect blueshift for the Tate construction of the $C_2$-fixed points of Lubin-Tate theory. It would also be interesting to investigate this question for fixed points of Lubin-Tate spectra with respect to larger subgroups of the Morava stabilizer group. 

\subsubsection{Tate blueshift for larger group analogs of $\e(n)$}
Recent work of Hahn--Shi \cite{HS17} and Beaudry--Hill--Shi--Zeng \cite{BHSZ20} produces genuine $C_{2^m}$-spectra $\e^m(n)$ (denoted as $BP^{(\!(C_{2^m})\!)}\langle n \rangle$ in \cite{BHSZ20}) for all $m \geq 1$ such that $\e^1(m) = \e(m)$. In other words, $\e^m(n)$ is an analog of Real Johnson--Wilson theory where $C_2$ is replaced by $C_{2^m}$. We also note that the chromatic height of $\e^m(n)$ is $n\cdot2^{m-1}$. 

We expect that there is a map of genuine $C_{2^m}$-spectra between a wedge of $RO(C_{2^m})$-graded suspensions of $\e^m(n-1)$ and $(\e^m(n)^{t_{C_{2^m}}\Sigma_2})_{\tilde{I}_{n-1}}^\wedge$ which becomes an equivalence of $C_{2^m}$-spectra after completion at an appropriate ideal $\tilde{I}_{n-1}$ obtained from $\bar{I}_{n-1}$. Moreover, we expect the $C_{2^m}$-fixed points of this equivalence to be a $K(n\cdot2^{m-1})$-local equivalence between a wedge of suspensions of $(\e^m(n-1))^{C_{2^m}}$ and $(\e^m(n)^{C_{2^m}})^{t\Sigma_2}$.

\subsubsection{Bounded below blueshift at higher heights}
One may recover connective real K-theory $ko$ as the $C_2$-fixed points of the slice cover of $\mathbb{E}(1)$. Slice covers do not generally interact well with Tate constructions, but perhaps the Davis--Mahowald splitting for $ko^{t\Sigma_2}$ can be recovered from our results. This would lead to a new splitting for the Tate constructions of connective covers of Real Johnson--Wilson theories. 

\subsubsection{$\k(n)$-local Tate vanishing}
Greenlees and Sadofsky's Tate vanishing for Morava K-theory \cite{GS96} was used to show that the Tate construction vanishes $K(n)$-locally by Hovey and Sadofsky in \cite{HS96}. An analogous result should hold for the parametrized Tate construction in the $\mathbb{K}(n)$-local setting. As in the classical setting, this $\mathbb{K}(n)$-local parametrized Tate vanishing ought to be an example of parametrized ambidexterity \cite[Sec. 4]{QS21a}. 

\subsection{Organization}\label{Sec:Organization}

In Section \ref{real}, we recall the results from Real oriented homotopy theory necessary for later sections. In particular, we recall the theory of Real representations, Real vector bundles, Real orientations, and the usual examples of Real oriented cohomology theories. We then recall several important results about Real Johnson--Wilson theory $\e(n)$, Real Morava K-theory $\k(n)$, and their $C_2$-fixed point spectra $ER(n)$ and $KR(n)$. 

In Section \ref{Sec:completionlocalization}, we study certain completions of the Real Johnson--Wilson spectrum $\e(n)$ and its fixed point spectrum $ER(n)$. We then discuss Bousfield localizations with respect to Real Morava K-theory $\k(n)$ and its fixed point spectrum $KR(n)$. The constructions are compared in Theorem \ref{ThmF} (Theorem \ref{prop:localizationcomparison}), which is used in the statement of Theorems \ref{ThmA} and \ref{ThmC}. 

In Section \ref{Sec:Tate}, we discuss variations of the Tate construction. We begin with a rapid review of the classical Tate construction and prove some useful results about its interaction with completion. We then turn to the ordinary and parametrized Tate constructions in the genuine equivariant setting. 

In Section \ref{blueshift}, we prove Theorems \ref{ThmA}, \ref{ThmC}, and \ref{ThmD}. We start with a quick outline of Ando, Morava, and Sadofsky's proof of Tate blueshift for (ordinary) Johnson--Wilson theories. We then explain how certain key ingredients of their proof can be upgraded to the genuine $C_2$-equivariant setting and prove Theorem \ref{ThmE} (Theorem \ref{Thm:ParamTateCoeffs}) and its analog for the ordinary Tate construction (Theorem \ref{Thm:TateCoeffsE}). This discussion culminates in the proof of Theorem \ref{ThmD} (Theorem \ref{Thm:Agree}), which is another key ingredient in the proof of Theorems \ref{ThmA} and \ref{ThmC}. With all of these ingredients in place, we then prove Theorems \ref{ThmC} (Theorem \ref{C2eqvtERsplit}) and \ref{ThmA} (Theorem \ref{ERtatesplitting}). 

In Section \ref{Sec:TateVanishing}, we turn to Tate vanishing. We begin by proving Theorem \ref{ThmRV} (Theorem \ref{Thm:TVMain}) by adapting the arguments of Greenlees--Sadofsky in the classical setting \cite{GS96} to the $C_2$-equivariant setting. Our argument is quite similar in spirit to theirs, but we require several new technical results. We then use the aforementioned results of Greenlees--Sadofsky in conjunction with the Kitchloo--Wilson cofiber sequence \cite{KW07} to prove Theorem \ref{ThmB} (Theorem \ref{Thm:realvanishing}); this requires a careful analysis of certain inverse limits.

\subsection{Notation}\label{Sec:Notation}
We will use bold face $\e$ to denote a $C_2$-equivariant spectrum and ordinary $E$ to denote its underlying nonequivariant spectrum. We will use $ER$ to denote its fixed points $ER=\e^{C_2}$. The $\e$-cohomology of a $C_2$-space $X$ is $RO(C_2)$-graded, and we will use $a+b\sigma$ to denote the direct sum of $a$ copies of the trivial representation with $b$ copies of the sign representation. We will use $\rho$ to denote the regular representation, $1+\sigma$. Classes in degrees $k\rho$ will generally have a bar over them, e.g. the classes $\overline{v}_i \in \pi_{(2^i-1)\rho}\e(n)$ above. Whenever a class in degree $k\rho$ is shifted into integer degree via multiplying by a power of the (invertible) class $y$ in degree $\lambda+\sigma$ (defined in Section \ref{sec:coefficients}), it ends up in degree $k(1-\lambda)$ and the result will have a hat over it, e.g. the classes $\widehat{v}_i \in \pi_{(2^i-1)(1-\lambda)}\e(n)$.

The ordinary $G$-Tate construction of a spectrum $X$ equipped with a $C_2$-action will be denoted $X^{tG}$. The $C_2$-parametrized $G$-Tate construction of a $C_2$-spectrum $X$ equipped with a $G$-action will be denoted $X^{t_{C_2}G}$. 

All limits taken in the category of spectra should be interpreted as homotopy limits, and the indexing category of $\limt{k}$ is always $\mathbb{N}^{\text{op}}$.

\subsection{Acknowledgments} 
The authors thank William Balderrama, Mark Behrens, Nitu Kitchloo, Doug Ravenel, Jay Shah, Nat Stapleton, Guozhen Wang, Dylan Wilson, Steve Wilson, Ningchuan Zhang, and Foling Zou for helpful discussions. The first author was supported by the Danish National Research Foundation through the Centre for Symmetry and Deformation (DNRF92) and the European Research Council (ERC) under the European Union’s Horizon 2020 research and innovation programme (grant agreement No 682922). The third author was partially supported by NSF grant DMS-1547292.

\section{Real representations and Real oriented cohomology theories}\label{real}
In this section, we discuss Real oriented homotopy theory. In Section \ref{Sec:RealReps}, we recall Real representation theory and discuss a special class of examples. We recall some definitions and results concerning Real oriented cohomology theories and especially the Real Johnson--Wilson theories in Sections \ref{Sec:RealCob} - \ref{sec:kwfibration}. 

\subsection{Real representations, spaces, and vector bundles}\label{Sec:RealReps}

We begin by recalling the theory of Real representations, Real spaces, and Real vector bundles. 

\begin{defin}\cite{AS69}\cite[Def. 2.2.1]{Fok14}
A \emph{Real Lie group} is a pair $(G,\sigma_G)$ where $G$ is a Lie group and $\sigma_G$ is a Lie group involution on it. A \emph{Real representation} $V$ of a Real Lie group $(G,\sigma_G)$ is a finite-dimensional complex representation of $G$ equipped with an anti-linear involution $\sigma_V$ such that $\sigma_V(gv) = \sigma_G(g) \sigma_V(v)$. 
\end{defin}

\begin{exm}\label{Exm:RealReps}
If $G$ is an abelian group, then inversion $(-)^{-1} : G \to G$ defines an involution on $G$ which equips $G$ with the structure of a Real Lie group. 
\end{exm}

We will also need the notion of Real spaces and Real vector bundles when we discuss Real orientations:

\begin{defin}\label{Def:RealVB}\cite[Def. 1.3.1]{Fok14}
A \emph{Real space} is a pair $(X,\sigma_X)$ where $X$ is a topological space equipped with an involutive homeomorphism $\sigma_X$. 

A \emph{Real vector bundle} over $X$ is a complex vector bundle $E$ over $X$ which itself is also a Real space with involutive homeomorphism $\sigma_E$ satisfying
\begin{enumerate}
\item $\sigma_X \circ p = p \circ \sigma_E$, where $p : E \to X$ is the projection map, and
\item $\sigma_E$ maps $E_x$ to $E_{\sigma_X(x)}$ anti-linearly. 
\end{enumerate}
\end{defin}

\subsection{Real cobordism and Real orientations}\label{Sec:RealCob}

We now recall the notion of Real orientation from \cite{HK01}. 

\begin{defin}\cite[Def. 2.2]{HK01}
A $C_2$-spectrum $\e$ is said to be \emph{Real oriented} if there is a class $x \in  \e^{\rho}(\mathbb{C}P^\infty)$ which restricts to 1 in $\e^\rho(\mathbb{C}P^1)=\e^\rho(S^\rho) \cong \e^0(S^0)$. 
\end{defin}

The \emph{Real cobordism spectrum} $\mathbb{MR}$ was first studied by Araki and Landweber; see \cite[Sec. 2]{HK01} for references. If $\e$ is a $C_2$-equivariant ring spectrum, Real orientability is equivalent to the existence of a map of ring spectra $\mathbb{MR} \longrightarrow \e$.

A Real oriented theory $\e$ has Thom isomorphisms for Real vector bundles. We would also like to record a general result from \cite{HM17} that we will find useful for proving when two $C_2$-spectra are weakly equivalent:

\begin{prop}\label{HMLemma}\cite[Lem. 3.4]{HM17}
Let $f: \e \longrightarrow \mathbb{F}$ be a natural transformation of $C_2$-equivariant homology theories with underlying theories $E$ and $F$. Assume that $f$ induces isomorphisms $\e_{k\rho} \longrightarrow \mathbb{F}_{k\rho}$ and $E_{2k} \longrightarrow F_{2k}$ for all $k \in \mathbb{Z}$. Assume further that $\e_{k\rho-1} \longrightarrow \mathbb{F}_{k\rho-1}$ is mono for all $k \in \mathbb{Z}$. Then $f$ is a natural isomorphism.
\end{prop}

We will always apply this proposition in the case where $\e$ and $\mathbb{F}$ satisfy the following:
\begin{enumerate}
\item Both $\e$ and $\mathbb{F}$ are $\mathbb{MR}$-modules,
\item The forgetful maps $\e_{k \rho}\longrightarrow E_{2k}$ and $\mathbb{F}_{k\rho} \longrightarrow F_{2k}$ are isomorphisms, and 
\item $\e_{k\rho-1}=\mathbb{F}_{k\rho-1}=0$. 
\end{enumerate}
To apply the proposition to $f : \e \to \mathbb{F}$, it then suffices to check that $f$ induces an isomorphism on coefficients in degrees $k \rho$. 

\subsection{Real Johnson--Wilson theories and Real Morava K-theories}\label{Sec:RJW}
We now recall the $C_2$-spectra which are the focus of this paper. In degrees $k \rho$, the forgetful map $\pi_{k\rho}\mathbb{MR}_{(2)} \longrightarrow \pi_{2k}MU_{(2)}$ is an isomorphism \cite{HK01}. It follows that there are $C_2$-equivariant refinements $\overline{v}_k \in \pi_{(2^k-1)\rho}\mathbb{MR}_{(2)}$ of the ordinary nonequivariant $v_k \in \pi_{2(2^k-1)}MU_{(2)}$. 

We let $\mathbb{MR}(n)$ denote the spectrum $\mathbb{MR}_{(2)}[\overline{v}_n^{-1}]$. It is a $C_2$-equivariant commutative ring spectrum \cite[Lemma 4.2]{KLW18}.

The \emph{$n$-th Real Johnson--Wilson theory} $\e(n)$ is constructed from $\mathbb{MR}(n)$ by killing the ideal generated by $\overline{v}_i$ for $i>n$. Its underlying spectrum is the $n$-th Johnson--Wilson theory $E(n)$ whose coefficients are $\pi_*E(n)=\mathbb{Z}_{(2)}[v_1, \dots, v_{n-1}, v_n^{\pm 1}]$. 

The \emph{$n$-th Real Morava K-theory} $\mathbb{K}(n)$ is obtained from $\e(n)$ by further quotienting by $\overline{v_i}$ for $0 \leq i \leq n-1$. The underlying spectrum of $\mathbb{K}(n)$ is the $n$-th Morava K-theory $K(n)$ whose coefficients are $\pi_*K(n)=\mathbb{F}_2[v_n^{\pm1}]$.

\subsubsection{Multiplicative structure} 
It is not known whether $\e(n)$ or $\mathbb{K}(n)$ are homotopy commutative, associative, and unital ring spectra. In the case of $\e(n)$, Kitchloo, Wilson, and the second author have shown in \cite{KLW17} that it is homotopy commutative, associative, and unital up to phantom maps. Moreover, $\e(n)$ represents a multiplicative cohomology theory on the category of spaces valued in commutative rings. This turns out to be enough structure for our purposes. We also note that the $C_2$-fixed points $ER(n)$ have an $E_\infty$-ring structure after $K(n)$-localization by work of Hahn and Shi \cite{HS17}. In situations where more structure is needed, the spectrum $\mathbb{MR}(n)$ is a good replacement for $\mathbb{E}(n)$.

Even less is known about the multiplicative structure of Real Morava K-theory. It is not known if $\mathbb{K}(n)$ is a homotopy commutative and associative ring spectrum, but we do know that the homotopy groups of its fixed points $\pi_*KR(n)$ cannot support a ring structure which is compatible with the inclusion of fixed points map $\pi_*KR(n) \longrightarrow \pi_*K(n)$. Indeed, this is visible even when $n=1$, i.e. for mod 2 real $K$-theory. One can check that there is a class in $\pi_2(KO/2)$ which maps to $v_1\in \pi_2(KU/2)$. This class must cube to zero in the source since $\pi_6(KO/2)=0$, yet $v_1^3 \neq 0$ in $\pi_6(KU/2)$. We thank Steve Wilson for pointing this out to us. We do however know that $\mathbb{K}(n)$ is a module spectrum over $\mathbb{MR}(n)$,  which is all that we will require.

\subsubsection{Freeness and cofreeness}

As a $C_2$-spectrum, $\e(n)$ has some nice properties. The natural maps $E{C_2}_+ \wedge \e(n) \longrightarrow \e(n)$ and $\e(n) \longrightarrow F(E{C_2}_+, \e(n))$ are both equivariant equivalences, making $\e(n)$ $C_2$-free and cofree (Definition \ref{Def:FreeCofree}), respectively \cite[Comment (4) on p. 349]{HK01}. These properties will be important to our computation of the Tate construction.

\subsection{The coefficients of $ER(n)$}\label{sec:coefficients}
The homotopy of $\e(n)$ may be computed by a homotopy fixed point spectral sequence, a Bockstein spectral sequence, or a slice spectral sequence (\cite[Section 4]{KW07}, \cite[Theorem 3.1]{KW15}, and \cite[Theorem 9.9]{HHR16}, respectively). In degrees multiples of the regular representation, the forgetful map gives an isomorphism $\pi_{k \rho} \e(n) \longrightarrow \pi_{2k}E(n)$. Outside of these degrees, the coefficients are considerably more interesting and contain some 2-torsion. The reader may consult \cite[Section 3]{KW15} for a full description. 

For our purposes, it will suffice to describe some distinguished classes in $\pi_{\star}\e(n)$ and $\pi_k \e(n)=\pi_k ER(n)$ for $n \geq 1$. Let $\lambda$ denote the positive integer $2^{n+2}(2^{n-1}-1)+1$.
 
 \begin{enumerate}
 \item There are equivariant refinements $\overline{v}_i \in \pi_{(2^i-1)\rho}\e(n)$ of the classical nonequivariant $v_i$ in $\pi_{2(2^i-1)}E(n)$.
 \item There is an invertible class $v_n^{2^{n+1}} \in \pi_{2^{n+1}(2^n-1)}\e(n)$ which makes the fixed-point theory $2^{n+1}(2^n-1)$-periodic. When $n=1$, this gives the $8$-periodicity of $KO$, and when $n=2$, we get the 48-periodicity of $ER(2)$.
 \item There is an invertible class $y \in \pi_{\lambda+\sigma} \e(n)$. By multiplying by a power of $y$, we may move any class into integer grading. In particular, given a class in degree a multiple of the regular representation $z \in \pi_{k \rho}\e(n)$, we may multiply by $y^{-k}$ to obtain $\widehat{z}:=zy^{-k} \in \pi_{(1-\lambda)k}ER(n)$.
 \item There is a class $x \in \pi_{\lambda}ER(n)$ which is 2-torsion (and generates all of the 2-torsion in the coefficients). We have $x^{2^{n+1}-1}=0$. When $n=1$, this class is $\eta \in \pi_1KO_{(2)}$.
 \end{enumerate}


Finally, we note that pieces of the Hurewicz image of $ER(n)$ can be computed using knowledge of the stable stems. Work of Shi, Wang, Xu, and the first author \cite[Thm. 1.8]{LSWX19} says that if the $i$-th Hopf invariant one element, the $i$-th Kervaire invariant one element, or the $j$-th $\bar{\kappa}$-element survives in the Adams spectral sequence for the sphere, then its image under the Hurewicz map $\pi_*(S^0) \to \pi_*(ER(n))$ is nontrivial when $i \leq n$ or $j \leq n-1$. 


\subsection{The Kitchloo--Wilson fibration}\label{sec:kwfibration}
Let $n \geq 1$. In \cite{KW07}, Kitchloo and Wilson show that multiplication by the class $x \in \pi_{\lambda}ER(n)$ yields a fibration of spectra
$$\Sigma^{\lambda}ER(n) \rightarrow ER(n) \rightarrow E(n).$$
In fact, for any $C_2$-spectrum $\e$, one can take the equivariant cofibration $C_{2_+} \rightarrow S^0 \rightarrow S^\sigma$ and derive from it the fibration
$$(\Sigma^{-\sigma}\e)^{C_2} =F(S^{\sigma}, \e)^{C_2} \rightarrow ER=\e^{C_2} \rightarrow F(C_{2_+}, \e)^{C_2} = E$$
in which the first map is multiplication by $a_{\sigma}$. If the spectrum $\e$ is further an $\mathbb{MR}(n)$-module spectrum, then the results of \cite{KW07} show that the invertible class $y$ described in the previous section is in fact an invertible class in $\pi_{\lambda+\sigma}\mathbb{MR}(n)$. Multiplication by this class gives an equivariant equivalence that allows one to shift the suspension in the first term into integer grading, yielding the Kitchloo--Wilson fibration
$$\Sigma^{\lambda}ER \rightarrow ER \rightarrow E$$
for \emph{any} $\mathbb{MR}(n)$-module spectrum $\mathbb{E}$. We will use this extensively when we study Bousfield localizations of $\e(n)$, $ER(n)$, and related spectra in Section \ref{Sec:completionlocalization}.

\subsection{$ER(n)$ orientations}
When identifying the Tate construction for $ER(n)$ in Section \ref{blueshift}, we will require the following result concerning when real vector bundles are $ER(n)$-orientable:

\begin{thm}\label{Thm:Or} \cite[Theorem 6.1]{KLW18}, \cite[Theorem 1.4]{KW15}
For any real vector bundle $\xi$, the bundle $2^{n+1}\xi$ is orientable (and thus has a Thom class) with respect to $ER(n)$. There is a corresponding $E(n)$-orientation such that the $ER(n)$-based Euler class for $2^{n+1}\xi$ is sent under the forgetful map to a unit (in the cohomology of the base) multiple of the $E(n)$-based Euler class.
\end{thm}

We note in passing that while the exponent of $2$ in this result cannot be improved (it is shown in \cite{KLW18} that there exists a bundle such that $2^n$ times it is not $ER(n)$-orientable), this result is certainly not optimal. For $n=1$, it states that $4\xi$ is $KO_{(2)}$-orientable for any $\xi$; however, the stronger statement that Spin bundles are $KO_{(2)}$-orientable is known. It remains an interesting open problem to identify criteria under which a bundle possesses an $ER(n)$-orientation.

\section{Completion and Bousfield localization}\label{Sec:completionlocalization}
In this section, we collect some results concerning completions and Bousfield localizations of $\mathbb{E}(n)$ and its fixed points that we will use in later sections. Some may be of independent interest.

\subsection{Completion of $\e(n)$}\label{Sec:Completion}
We begin by defining completion for $\mathbb{MR}$-module spectra and describing its effect on homotopy groups.

\begin{defin}\label{definition:completion}(Compare with \cite{GM97})
Let $a \in \pi_\star^{C_2} \mathbb{MR}$ and let $M$ be an $\mathbb{MR}$-module. We define the completion $M_a^\wedge$ by
$$M_a^\wedge := \limt{k} Ca^k,$$
where $Ca^k$ is the cofiber of the composite (where we suppress suspensions)
$$M\simeq S^0\wedge M \xrightarrow{unit \wedge id_M} \mathbb{MR} \wedge M \xrightarrow{a^k\wedge id_M} \mathbb{MR} \wedge M \rightarrow M.$$
Let $I$ be an ideal in $\pi_\star^{C_2} \mathbb{MR}$ generated by $(a_1,\cdots,a_k)$. We define $M_I^\wedge$ by setting
$$M^\wedge_I := ((M_{a_1}^\wedge)_{a_2}^\wedge \cdots)_{a_k}^\wedge.$$
\end{defin}

\begin{rem2}
It is straightforward to verify that 
$$(M_a^\wedge)_b^\wedge \simeq (M_b^\wedge)_a^\wedge.$$
\end{rem2}

\begin{lem}\label{lemma:completion}
Let $a$ and $M$ be as in Definition \ref{definition:completion}. If there exists an $N>0$ such that no nontrivial element in $\ker(\pi_\star^{C_2}M\xrightarrow{\cdot a} \pi_\star^{C_2}M)$ is $a^N$-divisible, then 
$$\pi_\star^{C_2}(M_a^\wedge) = (\pi_\star^{C_2}M)_a^\wedge.$$
\end{lem}

\begin{proof}
By the Milnor exact sequence for homotopy groups, the left-hand side
$$\pi_\star^{C_2}(M_a^\wedge) = \pi_\star^{C_2}(\limt{k} Ca^k) $$
fits into a short exact sequence
$$0 \rightarrow \text{lim}^1 \pi_{\star+1}^{C_2}(Ca^k) \rightarrow \pi_\star^{C_2}(M_a^\wedge) \rightarrow \limt{k} \pi_\star^{C_2}(Ca^k) \rightarrow 0.$$
Consider the diagram
\begin{center}
\begin{tikzcd}
0 \arrow[r] & \pi_{\star+1}^{C_2}M/a^{k+1}\pi_{\star+1}^{C_2}M \arrow[r] \arrow[d] & \pi_{\star+1}^{C_2}(Ca^{k+1}) \arrow[r] \arrow[d]& \ker(\pi_\star^{C_2}M \xrightarrow{\cdot a^{k+1}} \pi_\star^{C_2}M) \arrow[r] \arrow[d,"\cdot a"]& 0 \\

0 \arrow[r] & \pi_{\star+1}^{C_2}M/a^k\pi_{\star+1}^{C_2}M \arrow[r] & \pi_{\star+1}^{C_2}(Ca^k) \arrow[r] & \ker(\pi_\star^{C_2}M \xrightarrow{\cdot a^k} \pi_\star^{C_2}M) \arrow[r] & 0. 
\end{tikzcd}
\end{center}
The left vertical map is always surjective, so by \cite[\href{https://stacks.math.columbia.edu/tag/0598}{Tag 0598}]{stacks-project}, we have a short exact sequence
$$0 \rightarrow \limt{k} \pi_{\star+1}^{C_2}M/a^k\pi_{\star+1}^{C_2}M \rightarrow \limt{k} \pi_{\star+1}^{C_2}(Ca^k) \rightarrow \limt{k} \ker(\pi_\star^{C_2}M \xrightarrow{\cdot a^k} \pi_\star^{C_2}M) \rightarrow 0.$$
Note that
$$\limt{k} \pi_{\star+1}^{C_2}M/a^k\pi_{\star+1}^{C_2}M= (\pi_{\star+1}^{C_2}M)_a^\wedge$$
and
$$\limt{k} \ker(\pi_\star^{C_2}M \xrightarrow{\cdot a^k} \pi_\star^{C_2}M) = 0$$
by the assumption that $\pi_\star^{C_2}M$ has no infinitely $a$-divisible element. Hence, we have 
$$\limt{k} \pi_\star^{C_2}(Ca^k)=(\pi_{\star+1}^{C_2}M)_a^\wedge.$$
The diagram also shows that $\{\pi_\star^{C_2}(Ca^k)\}$ is Mittag-Leffler, so we only need to show the right-hand inverse system is Mittag-Leffler. In fact, the stable image is $0$. Indeed, for all nontrivial elements $x \in \ker(\pi_\star^{C_2}(M) \xrightarrow{\cdot a^k} \pi_\star^{C_2}(M))$, there is no $y \in \ker(\pi_\star^{C_2}M \xrightarrow{\cdot a^{k+N}} \pi_\star^{C_2}M)$ satisfying $x=a^Ny$ by assumption, so $x$ is not in the image of $\ker(\pi_\star^{C_2}M \xrightarrow{\cdot a^{k+n}} \pi_\star^{C_2}M)$ for $n>N$.
\end{proof}

Now consider the ideal $\bar{I}_n:=(\vbar_0, \vbar_1, \dots, \vbar_{n-1})$. We have that the homotopy groups of the spectrum $\e(n)^\wedge_{\bar{I}_n}$ are exactly the $\bar{I}_n$-adic completion of $\pi_\star \e(n)$. We use this to prove a $C_2$-equivariant analog of a proposition of Ando, Morava, and Sadofsky \cite[Section 3]{AMS98}

\begin{prop}\label{modulestructure} There is a $C_2$-equivalence
$$(\bar{v}_{n-1}^{-1}BP\r)^\wedge_{\bar{I}_{n-1}} \simeq \left( \bigvee_{R \in \mathcal{R}} \Sigma^{|\sigma R|}\e(n-1) \right)_{\bar{I}_{n-1}}^{\wedge}.$$
In particular, $(\bar{v}_{n-1}^{-1}BP\r)^\wedge_{\bar{I}_{n-1}}$ is a module spectrum over $\e(n-1)$.\end{prop}

\begin{proof}
Nonequivariantly $(v_{n-1}^{-1}BP)^\wedge_{I_{n-1}}$ is an $E(n-1)$-module. As in \cite{AMS98} ($p=2$), the module structure is given as follows. Let $R=(r_1,r_2,\cdots)$ range over multi-indices of nonnegative integers (with only finitely many positive coordinates) and let 
$$|R|=2(r_1(2-1)+r_2(2^2-1)+\cdots).$$
Let $\mathcal{R}$ be the set of multi-indices with the first $n-2$ indices $0$, and set
$$\sigma R=(2^{n-1}r_{n-1}, 2^{n-1}r_n,\cdots).$$
Let $q$ be the $BP_*$-module quotient map corresponding to 
$$BP_*BP=BP_*[t_1,t_2,\cdots] \rightarrow BP_*[t_1^{2^{n-1}},t_2^{2^{n-1}},\cdots].$$
Then the composite $BP$-module map
\begin{equation}\label{equation:3.4}
BP \xrightarrow{\eta_R} BP \wedge BP \xrightarrow{\sim} \bigvee_{R \in \mathcal{R}} \Sigma^{|R|}BP \xrightarrow{q}  \bigvee_{R \in \mathcal{R}} \Sigma^{|\sigma R|}BP \xrightarrow{\theta}  \bigvee_{R \in \mathcal{R}} \Sigma^{|\sigma R|}BP\langle n-1 \rangle
\end{equation}
is a homotopy equivalence after inverting $v_{n-1}$ and completing at $I_{n-1}$. The right-hand side becomes
$$\left( \bigvee_{R \in \mathcal{R}} \Sigma^{|\sigma R|}E(n-1) \right)^\wedge_{I_{n-1}},$$
which admits an $E(n-1)$-module structure. Hence $(v_{n-1}^{-1}BP)^\wedge_{I_{n-1}}$ is an $E(n-1)$-module.


From \cite[Thm. 4.11]{HK01}, we have a $C_2$-equivariant lift of the map (\ref{equation:3.4}), in which 
$$|R|=(r_1(2-1)+r_2(2^2-1)+\cdots)\rho$$
now lies in in $RO(C_2)$. After applying $\bar{v}_{n-1}(-/\bar{I}_n^k)$, both sides are strongly even and the underlying map is a homotopy equivalence, so it is a $C_2$-equivalence by Proposition \ref{HMLemma}. From here, the $C_2$-equivariant case proceeds as in the nonequivariant setting.
\end{proof}

\begin{rem2}
Proposition \ref{modulestructure} states that after $K(n)$-localization, $\bar{v}_n^{-1}BP\mathbb{R}$ splits as a wedge of $RO(C_2)$-graded suspensions of $\e(n)$. It is an interesting open question whether an analogous splitting holds for cyclic groups of order $2^m$, where $\e(n)$ is replaced by a certain genuine $C_{2^m}$-spectrum constructed by Beaudry--Hill--Shi--Zeng in \cite{BHSZ20}. 
\end{rem2}


\subsection{Bousfield localization of $\e(n)$}\label{Sec:Localization}
We will be interested both in the $C_2$-equivariant Bousfield localization $L_{\mathbb{K}(m)}\mathbb{E}(n)$ as well as the Bousfield localization involving the fixed points $L_{KR(m)}ER(n)$. We refer the reader to \cite{Bou79} for Bousfield localization, \cite{Rav84} for chromatic Bousfield localizations, and \cite{Hov07} for Bousfield localization in model categories. 

We begin by reviewing free and cofree spectra.

\begin{defin}\label{Def:FreeCofree}
A $C_2$-spectrum $\e$ is \emph{$C_2$-free} if the natural map ${EC_2}_+ \wedge \e \to \e$ is a $C_2$-equivariant equivalence. We say $\e$ is \emph{$C_2$-cofree} if the natural map $\e \to F({EC_2}_+,\e)$ is a $C_2$-equivariant equivalence. 
\end{defin}

\begin{lem}\label{freecofreecomparison} If $f: X \rightarrow Y$ is an equivariant map which is an underlying equivalence, then
\begin{enumerate}
\item if $\mathbb{E}$ is $C_2$-free, then $\mathbb{E} \wedge X \rightarrow \mathbb{E} \wedge Y$ is an equivariant equivalence, and
\item if $\mathbb{E}$ is $C_2$-cofree, then $F(Y, \mathbb{E}) \rightarrow F(X, \mathbb{E})$ is an equivariant equivalence.
\end{enumerate}
\end{lem}
\begin{proof}Since $X \longrightarrow Y$ is a $C_2$-equivariant map which is an underlying nonequivariant equivalence, the induced map
$$X \wedge {EC_2}_+ \longrightarrow Y \wedge {EC_2}_+$$
is a $C_2$-equivariant equivalence. Smashing with $\e$ and mapping into $\e$, respectively, yield maps
\begin{align*}X \wedge {EC_2}_+ \wedge \e &\longrightarrow Y \wedge {EC_2}_+ \wedge \e,\\
F(Y, F({EC_2}_+, \e)) \simeq F({Y \wedge EC_2}_+ , \e) &\longrightarrow F({Y \wedge EC_2}_+ , \e) \simeq F(Y, F({EC_2}_+, \e))\end{align*}
which are also $C_2$-equivariant equivalences. If $\e$ is free, the canonical map ${EC_2}_+ \wedge \e \longrightarrow \e$ is an equivalence. If $\e$ is cofree, then the canonical map $\e \longrightarrow F({EC_2}_+, \e)$ is an equivalence. 
\end{proof}

Recall that $\mathbb{E}(n)$ and $\mathbb{K}(n)$ are both free and cofree. This has the following consequence.

\begin{cor} An equivariant map $X \rightarrow Y$ is a $\mathbb{K}(m)$- (resp. $\mathbb{E}(m)$)-equivalence if and only if the underlying nonequivariant map is a $K(m)$- (resp. $E(m)$-) equivalence. An equivariant spectrum $X$ is $\mathbb{K}(m)$- (resp. $\mathbb{E}(m)$)-acyclic if and only if its underlying nonequivariant spectrum is $K(m)$- (resp. $E(m)$-)acyclic.
\end{cor}

We state one more lemma, which is true in both the equivariant and nonequivariant category:

\begin{lem} \label{lemma:BousfieldClass}
If 
$$\Sigma^V A \xrightarrow{a} A \rightarrow B$$
is a cofiber sequence of spectra in which the map $a$ is nilpotent, then $A$ and $B$ have the same homological and cohomological Bousfield classes: $\langle A \rangle = \langle B \rangle$ and $\langle A^* \rangle=\langle B^* \rangle$. 
\end{lem}
\begin{proof}
If $C \in \langle A \rangle$, then $C \wedge A \simeq *$. The cofiber sequence
$$C \wedge A \xrightarrow{\text{id}_C \wedge a} C \wedge A \rightarrow C \wedge B$$
implies that $C \wedge B \simeq *$ and $C \in \langle B \rangle$.

If $C \in \langle B \rangle$, then $C \wedge B \simeq *$. The cofiber sequence
$$C \wedge A \xrightarrow{\text{id}_C \wedge a} C \wedge A \rightarrow C \wedge B$$
implies that the self map $\text{id}_C \wedge a$ is a weak equivalence. Because $a$ is nilpotent, so is $\text{id}_C \wedge a$. Therefore, $C \wedge A \simeq *$ and $C \in \langle A \rangle$.

Since mapping into a cofiber sequence of spectra also gives a long exact sequence, the corresponding result for cohomological Bousfield classes follows along the same lines. 
\end{proof}

\begin{prop}\label{prop:bousfieldclasses} If $\mathbb{F}$ is any spectrum on which $a_{\sigma}\in \pi_\sigma \mathbb{S}$ acts nilpotently, then $\mathbb{F}$ and $F(C_{2_+}, \mathbb{F})$ have the same homological and cohomological Bousfield classes. If we further have that $\mathbb{F}$ is an $\mathbb{MR}(n)$-module spectrum (for instance if $\mathbb{F}$ is $\mathbb{E}(m)$ or $\mathbb{K}(m)$), then $F$ and $FR$ have the same (homological and cohomological) Bousfield classes. \end{prop}
\begin{proof} Applying the previous lemma to the cofiber sequence
$$\xymatrix{ \Sigma^{-\sigma} \mathbb{F} \ar@{->}[r]^-{a_{\sigma}} &\mathbb{F} \ar@{->}[r] & F(C_{2_+}, \mathbb{F}) }$$
proves the first statement. For the second statement, note that the coefficients of $\mathbb{MR}_{(2)}$ contain the invertible class $y$.  Multiplying by a power of this class to shift the $\sigma$-suspension into integer degree and taking fixed points yields the cofiber sequence
$$\xymatrix{\Sigma^{\lambda} FR  \ar@{->}[r]^-{x} & FR \ar@{->}[r] & F}.$$
Applying the previous lemma to this cofiber sequence yields the second statement.
\end{proof}

\begin{cor} \label{cor:local} Let $\mathbb{X}$ be a $C_2$-spectrum on which $a_{\sigma}$ acts nilpotently. For any $C_2$-spectrum $\mathbb{F}$, if $X$ is $F$-local, then $\mathbb{X}$ if $\mathbb{F}$-local. If we further have that $\mathbb{X}$ is an $\mathbb{MR}(n)$-module spectrum, then if $X$ is $F$-local, then $XR$ is $F$-local.
\end{cor}
\begin{proof} Let $\mathbb{Y}$ be $\mathbb{F}$-acyclic. Then for underlying spectra, we have that $Y$ is $F$-acyclic. By the previous proposition, $\mathbb{X}$ and $ F(C_{2_+}, \mathbb{X})$ have the same cohomological Bousfield class, so it suffices to show that $F(\mathbb{Y}, F(C_{2_+}, \mathbb{X}))^{C_2} \simeq \ast$. But we have
$$F(\mathbb{Y}, F(C_{2_+}, \mathbb{X}))^{C_2}\simeq F(C_{2_+}, F(\mathbb{Y}, \mathbb{X}))^{C_2} \simeq F(Y, X)$$
and the right hand side is contractible since $Y$ is $F$-acyclic and $X$ is $F$-local.
\end{proof}

\subsection{Comparing completion and Bousfield localization}
Recall that we have ideals $\overline{I}_m=(2, \overline{v}_1, \dots, \overline{v}_{m-1})$ of $\pi_{\star}\mathbb{MR}(m)$ and $\widehat{I}_m=(2, \widehat{v}_1, \dots, \widehat{v}_{m-1})$ of $\pi_*MR(m)$. Note that if we view $\pi_*MR(m)$ as a subring of $\pi_{\star}\mathbb{MR}(m)$, the inclusion of the ideal $\widehat{I}_m$ generates the ideal $\overline{I}_m$. In this section, we show for a class of spectra that includes $\e(n)$ as well as the parametrized Tate construction $\e(n)^{t_{C_2}\Sigma_2}$ (to be defined in the next section) that $\overline{I}_n$-completion and $\mathbb{K}(n)$-localization are equivalent. Similarly, we show that for a class of spectra that includes the fixed points $ER(n)$, as well as the Tate construction on the fixed points $ER(n)^{t\Sigma_2}$, that $\widehat{I}_n$-completion and $K(n)$-localization are equivalent.

\begin{thm}\label{prop:localizationcomparison} Let $\e$ be any $\mathbb{MR}(m)$-module spectrum such that the underlying canonical map $E \rightarrow E^{\wedge}_{I_m}$ factors through $L_{K(m)}E$ and gives an equivalence $L_{K(m)}E \simeq E^{\wedge}_{I_m}$. 
 \begin{enumerate}
\item The canonical equivariant map $\e \rightarrow \e^{\wedge}_{\overline{I}_m}$ factors through $L_{\k(m)}\e$ and gives an equivalence $L_{\k(m)}\e \simeq \e^{\wedge}_{\overline{I}_m}$.
\item 
The canonical map $ER \rightarrow ER^{\wedge}_{\widehat{I}_m}$ factors through $L_{K(m)}ER$ and gives an equivalence $L_{K(m)}ER \simeq ER^{\wedge}_{\widehat{I}_m}$.
\end{enumerate}
\end{thm}

\begin{proof} 
\begin{enumerate}
\item Since on underlying spectra we have $E^{\wedge}_{I_m} \simeq L_{K(m)}E$, we have by Corollary \ref{cor:local} that $\e^{\wedge}_{\overline{I}_m}$ is $\mathbb{K}(m)$-local. Thus, $\e^{\wedge}_{\overline{I}_m} \simeq L_{\mathbb{K}(m)}(\e^{\wedge}_{\overline{I}_m})$. Since the underlying nonequivariant map $E \rightarrow E^{\wedge}_{I_m}\simeq L_{K(m)}E$ is a $K(m)$-equivalence, it follows that the equivariant map $\e \rightarrow \e^{\wedge}_{\overline{I}_m}$ is a $\mathbb{K}(m)$-equivalence. Thus, we have
$$\e^{\wedge}_{\overline{I}_m} \simeq L_{\mathbb{K}(m)}(\e^{\wedge}_{\overline{I}_m}) \simeq L_{\mathbb{K}(m)}\e.$$
\item Again, we first apply Corollary \ref{cor:local} to conclude that $ER^\wedge_{\widehat{I}_m}$ is $K(m)$-local.


Now, denote the cofiber of $ER \rightarrow ER^\wedge_{\widehat{I}_m}$ by $C_{ER}$ and the cofiber of $E\rightarrow E^\wedge_{\widehat{I}_m}$ by $C_E$. Because completion and taking cofibers preserves cofiber sequences, the cofiber sequence (which exists for any $\mathbb{MR}(m)$-module spectrum $\e$)
$$\Sigma^{\lambda}ER \xrightarrow{x} ER \rightarrow E$$
gives a cofiber sequence
$$\Sigma^{\lambda}C_{ER} \xrightarrow{x} C_{ER} \rightarrow C_E.$$
Lemma \ref{lemma:BousfieldClass} applies to this cofiber sequence, so we have $\langle C_{ER} \rangle = \langle C_{E} \rangle$. We know $E^\wedge_{\widehat{I}_m} = E^\wedge_{I_m} = L_{K(m)}E$ so $C_E$ is $K(m)$-acyclic. Hence, $K(m) \in \langle C_{E} \rangle = \langle C_{ER} \rangle$. The cofiber $C_{ER}$ is $K(m)$-acyclic, so the map
$L_{K(m)}ER \rightarrow L_{K(m)}ER^\wedge_{\widehat{I}_m}$ is a $K(m)$-equivalence.
From the above lemmas, we have
$$ER^\wedge_{\widehat{I}_m} \simeq L_{K(m)} ER^\wedge_{\widehat{I}_m} \simeq L_{K(m)} ER.$$
\end{enumerate}
\end{proof}

\begin{cor} Let $\e = \e(n)$. We have equivalences
\begin{align*}
L_{\k(m)}\e &\simeq (\bar{v}_m^{-1}\e)^{\wedge}_{\overline{I}_m},\\
L_{K(m)}ER & \simeq (\widehat{v}_m^{-1}ER)^{\wedge}_{\widehat{I}_m}.
\end{align*}
In particular, when $m=n$ we have
\begin{align*}
L_{\k(n)}\e(n) &\simeq \e(n)^{\wedge}_{\overline{I}_n},\\
L_{K(n)}ER(n) & \simeq ER(n)^{\wedge}_{\widehat{I}_n}.
\end{align*}
\end{cor}

Note also that Proposition \ref{prop:bousfieldclasses} tells us that $KR(m)$ and $K(m)$ have the same Bousfield class. It follows that $L_{K(m)}(-)=L_{KR(m)}(-)$.

\section{Tate constructions}\label{Sec:Tate}
In this section, we recall three variations of the Tate construction. Each will be defined as the cofiber of a norm map between homotopy fixed points and homotopy orbits which depend on the category of the input and the $C_2$-action on the relevant universal spaces. 

\begin{itemize}
\item $E^{t\Sigma_2}$, \emph{the classical Tate construction for a nonequivariant spectrum} $E$ \emph{equipped with a trivial $\Sigma_2$-action}. This will be defined using the universal space $E\Sigma_2$ and may be identified with
$$E^{t \Sigma_2}\simeq \limt{n} RP_{-n}^\infty \wedge \Sigma E.$$
We discuss this further in Section \ref{SS:ClassicalTate}.

\item $\e^{t\Sigma_2}$, \emph{the classical Tate construction for a $C_2$-equivariant spectrum} $\e$ \emph{equipped with a trivial $\Sigma_2$-action}. This will be defined using the universal space $E\Sigma_2$ equipped with a trivial $C_2$-action, so $C_2$ will only act nontrivially on the spectrum $\e$. It can also be identified with 
$$\mathbb{E}^{t\Sigma_2} \simeq \limt{n} RP_{-n}^\infty \wedge \Sigma \mathbb{E}$$
with $C_2$ acting trivially on $RP_{-n}^\infty$. The resulting Tate construction will be a $C_2$-equivariant spectrum. We discuss this further in Section \ref{SS:TateC2}.

\item $\e^{t_{C_2}\Sigma_2}$, \emph{the parametrized Tate construction for a $C_2$-equivariant spectrum} $\e$ \emph{equipped with a trivial $\Sigma_2$-action}. This will be defined using a $C_2$-equivariant universal space $E_{C_2}\Sigma_2$ which has a nontrivial $C_2$-action, so $C_2$ will act nontrivially on on the universal space and on the spectrum. This may be identified with
$$\e^{t_{C_2}\Sigma_2}:=\limt{n} Q_{-n}^\infty \wedge \Sigma \e.$$
Here $Q_{-n}^\infty$ is nonequivariantly equivalent to $RP_{-n}^\infty$ but it carries a nontrivial $C_2$-action. We discuss this further in Sections \ref{SS:ParamTate} and \ref{Sec:ParamSigma2}.
\end{itemize}

We discuss the interaction between the Tate construction and completions in Section \ref{SS:TateCompletion}, and we compare the ordinary and parametrized Tate constructions in Section \ref{SS:Compare}. 



\subsection{The classical Tate construction}\label{SS:ClassicalTate}
Let $E$ be a nonequivariant spectrum. Let $U$ be a complete $\Sigma_2$ universe and let $i: U^{\Sigma_2} \longrightarrow U$ be the inclusion of fixed points.  


Let $E\Sigma_2$ denote the free contractible $\Sigma_2$-space and let $\widetilde{E\Sigma_2}$ denote the cofiber in the sequence 
$${E\Sigma_2}_+ \longrightarrow S^0 \longrightarrow \widetilde{E\Sigma_2}.$$
Then we have the \emph{$\Sigma_2$-Tate diagram}:
$$\xymatrix{
{E\Sigma_2}_+ \wedge i_*E \ar@{->}[r] \ar@{->}[d]^-{\simeq} & i_*E \ar@{->}[r] \ar@{->}[d] & \widetilde{E\Sigma_2} \wedge i_*E \ar@{->}[d] \\
{E\Sigma_2}_+ \wedge F({E\Sigma_2}_+, i_*E) \ar@{->}[r] & F({E\Sigma_2}_+, i_*E) \ar@{->}[r] & \widetilde{E\Sigma_2} \wedge F({E\Sigma_2}_+, i_*E)
}$$

\begin{defin} We define the $\Sigma_2$-Tate construction for a nonequivariant spectrum $E$ to be the lower right corner of this diagram. That is,
$$E^{t\Sigma_2}:=\widetilde{E\Sigma_2} \wedge F({E\Sigma_2}_+, i_*E).$$
\end{defin}

This may be identified with a homotopy limit:
\begin{prop}\cite[Thm. 16.1]{GM95} There is an equivalence of spectra
$$E^{t\Sigma_2} \simeq \limt{n} RP_{-n}^\infty \wedge \Sigma E.$$
 \end{prop}

In the definition of the Tate construction, we may write the space $\widetilde{E\Sigma_2}$ as a colimit and this is useful for computations. More precisely, let $V$ denote the $\Sigma_2$ sign representation. Then $\widetilde{E\Sigma_2} =\colim_n S^{nV}$.

Let $e: S^0 \longrightarrow S^V$ be the natural inclusion, and let $\alpha_V \in E^V(S^0)$ be the image of the identity element under this map. Let $c(E):=F({E\Sigma_2}_+, i_*E)$. When $E$ is a (homotopy) ring spectrum, the homotopy groups of the Tate construction may be identified with a localization of $c(E)$ away from $\alpha_V$:

\begin{prop}\cite[Cor. 16.3]{GM95} Let $E$ be a ring spectrum (which we view as a $\Sigma_2$-spectrum with trivial action). Then $\pi_*E^{t\Sigma_2}$ is the localization of $\pi_*c(E)=\pi_*E^{h\Sigma_2}$ away from $\alpha_V$. 
\end{prop}

When the bundle associated to the representation $V$ is orientable with respect to $E$, the above proposition may be simplified further. We will return to this point when we discuss the Tate construction for $E(n)$ in Section \ref{blueshift}.

\subsection{The classical Tate construction for a $C_2$-equivariant spectrum}\label{SS:TateC2}
Now we take $\e$ to be a $C_2$-spectrum. We let $U$ denote a complete $\Sigma_2 \rtimes C_2$ universe. Then $U^{\Sigma_2}$ is a complete $C_2$ universe and we let $i: U^{\Sigma_2} \longrightarrow U$ be the inclusion.

\begin{defin}
We define the \emph{$\Sigma_2$-Tate construction of $\e$} to be the $C_2$-spectrum
$$\e^{t\Sigma_2}:=\widetilde{E\Sigma_2}\wedge F(E{\Sigma_2}_+, i_*\e).$$
\end{defin}

The results concerning the ordinary Tate construction in the previous section carry over \emph{mutatis mutandis}. In particular, we have the same Tate diagram, now in the category of $(\Sigma_2 \rtimes C_2)$-spectra. Moreover, we may calculate the $\Sigma_2$-Tate construction of a $C_2$-spectrum via homotopy inverse limit. This follows from observing that each equivalence in the proof of \cite[Thm. 16.1]{GM95} may be promoted to an equivalence of $C_2$-spaces or genuine $C_2$-spectra when we assume that $C_2$ acts trivially on $EG$ and $V$. 

\begin{prop}\label{Prop:TrivTateForm}
Let $\e$ be a genuine $C_2$-spectrum equipped with a trivial $\Sigma_2$-action. There is an equivalence of genuine $C_2$-spectra
$$\e^{t\Sigma_2} \simeq \limt{n} RP^\infty_{-n} \wedge \Sigma \e.$$
\end{prop}

%
 


We can identify the $C_2$-fixed points of the classical $\Sigma_2$-Tate construction of $\e$ with the classical Tate construction of the fixed points of $\e$. 

\begin{prop}\label{Tatefixed} For any $C_2$-spectrum $\e$ (e.g. for $\e=\e(n)$) we have
$$\left(\e^{t\Sigma_2}\right)^{C_2}=ER^{t\Sigma_2}.$$
\end{prop}
\begin{proof} We write the Tate construction as a homotopy limit and take fixed points:
$$\left(\e^{t\Sigma_2}\right)^{C_2}=\left(\limt{k} \e \wedge RP_{-k}^\infty\right)^{C_2}=\limt{k} (\e \wedge RP_{-k}^\infty)^{C_2}.$$
The claim will follow after we show that for any space $X$ with trivial $C_2$-action (or integral suspension thereof) that $(\e \wedge X)^{C_2}=ER \wedge X$. Indeed, compare the top rows of the Tate diagrams:
$$\xymatrix{
(E{C_2}_+ \wedge \e \wedge X)^{C_2}  \ar@{->}[r] & (\e \wedge X)^{C_2}  \ar@{->}[r] & (\widetilde{EC_2} \wedge \e \wedge X)^{C_2}  \\
(E{C_2}_+ \wedge \e)^{C_2} \wedge X \ar@{->}[u] \ar@{->}[r] & \e^{C_2} \wedge X \ar@{->}[u] \ar@{->}[r] & (\widetilde{EC_2} \wedge \e)^{C_2} \wedge X. \ar@{->}[u]}$$
Identifying the terms in the right column with geometric fixed points, we see that the right-hand map is an equivalence since geometric fixed points commute with smash products. To see the left-hand map is an equivalence, we apply the Adams isomorphism and note that $(E{C_2}_+ \wedge \e \wedge X)/C_2 =(E{C_2}_+ \wedge \e)/{C_2} \wedge X$ since $X$ has trivial action. It follows that the middle map is an equivalence.
\end{proof}

\subsection{Tate construction and completion}\label{SS:TateCompletion}
We show that under certain conditions, completion and the Tate construction commute. This will help us deduce blueshift for $E^\wedge(n)^{t\Sigma_2}$ from blueshift for $E(n)^{t\Sigma_2}$.

\begin{lem}\label{lemma:algebra}
Let $R$ be a commutative ring, $A$ be an $R$-algebra, $r \in R$. Then
$$A^\wedge_r/(x) \cong (A/(x))^\wedge_r.$$
\end{lem}

\begin{proof}
Let $K_k=\ker(A/r^k \xrightarrow{x} A/r^k)$, $A_k=A/(r^k,K_k)$, $B_k=A/r^k$ and $C_k=A/(r_k,x)$. The exact sequences
$$0\rightarrow K_k \rightarrow B_k \xrightarrow{x} B_k \rightarrow C_k \rightarrow 0,$$
split into short exact sequences
$$0\rightarrow K_k \rightarrow B_k \xrightarrow{x} A_k \rightarrow 0,$$
$$0\rightarrow A_k \rightarrow B_k \rightarrow C_k \rightarrow 0.$$
The map $A_{k+1} \to A_k$ is surjective since the quotient maps $B_{k+1}\rightarrow B_k$ are surjective, and surjectivity of the map $K_{k+1} \rightarrow K_k$ follows by definition. Therefore 
$\{K_k\}$ and $\{A_k\}$ satisfies the Mittag-Leffler condition and we have an exact sequence
$$0 \rightarrow \holim K_k \rightarrow \holim B_k \xrightarrow{x} \holim B_k \rightarrow \holim C_k \rightarrow 0.$$
This implies that
$$A^\wedge_r/(x)=(\holim B_k)/(x)=\holim C_k= (A/x)^\wedge_r.$$ 
\end{proof}

\begin{prop}\label{prop:tatecompletion}
\begin{enumerate}
\item
Let $a$ be an element of $\pi_*MU$, and $X$ be a $MU$-algebra equipped with trivial $\Sigma_2$ action. Assume that there exists $N>0$ such that no element in the kernel of
$$a_* \colon X_* \longrightarrow X_*,$$
or the kernel of
$$a_* \colon \pi_*(X^{t\Sigma_2}) \longrightarrow \pi_*(X^{t\Sigma_2}),$$
is $a^N$-divisible. Then
$$(X^{t\Sigma_2})_a^\wedge=(X_a^\wedge)^{t\Sigma_2}.$$
\item
Let $\alpha$ be an element of $\pi_\star \mathbb{MR}$, and $\mathbb{X}$ be a $\mathbb{MR}$-algebra equipped with trivial $\Sigma_2$ action. Assume that there exists $N>0$ such that no element in the kernel of
$$\alpha_\star \colon \mathbb{X}_\star \longrightarrow \mathbb{X}_\star,$$
or the kernel of
$$\alpha_\star \colon \pi_\star (\mathbb{X}^{t\Sigma_2}) \longrightarrow \pi_\star (\mathbb{X}^{t\Sigma_2}),$$
is $\alpha^N$-divisible. Then
$$(\mathbb{X}^{t\Sigma_2})_\alpha^\wedge=(\mathbb{X}_\alpha^\wedge)^{t\Sigma_2}.$$
\end{enumerate}
\end{prop}

\begin{proof}
We prove (1); the proof of (2) is similar. We will construct a map $f$ from $(X^\wedge_a)^{t\Sigma_2}$ to $(X^{t\Sigma_2})^\wedge_a$ and show
that the induced map $f_*$ on homotopy groups is an isomorphism. Smashing the qoutient maps
$$q_k \colon X^\wedge_a \longrightarrow X/a^k$$
with $RP^\infty_{-i}$, we have maps
$$f_{i,k} \wedge \id \colon X^\wedge_\alpha \wedge RP^\infty_{-i} \rightarrow X/a^k \wedge RP^\infty_{-i} \simeq (X \wedge RP^\infty_{-i})/a^k.$$
Therefore, there exists
$$f_i \colon X^\wedge_a \wedge RP^\infty_{-i} \rightarrow \lim (X \wedge RP^\infty_{-i})/a^k = (X \wedge RP^\infty_{-i})^\wedge_a.$$
These maps induce a map
$$f \colon \holim(X^\wedge_a \wedge RP_{-i}^\infty) \rightarrow \holim((X \wedge RP^\infty_{-i})^\wedge_a).$$
By definition, the completion is a composition of a finite colimit followed by a limit. The category of spectra is stable and finite colimits are also finite limits. In particular, the completion commutes with all limits. Therefore, we have
$$\holim((X \wedge RP^\infty_{-i})^\wedge_a)=(\holim(X \wedge RP^\infty_{-i}))^\wedge_a=(X^{t\Sigma_2})^\wedge_a.$$
Now we have
$$f \colon (X^\wedge_a)^{t\Sigma_2}\longrightarrow (X^{t\Sigma_2})^\wedge_a.$$
By the assumption, we can apply Lemma \ref{lemma:completion} and have
$$\pi_*(X_a^\wedge)=(\pi_*X)_a^\wedge.$$
Note that the $MU$-module structure gives a formal group $F$ over $\pi_*X$. This allows us to compute that
$$\pi_*(X^\wedge_a)^{t\Sigma_2}=\pi_*(X^\wedge_a)((x))/[2]_F(x)
=(\pi_*X)^\wedge_a((x))/[2]_F(x).
$$
Similarly, we have
$$
\pi_*((X^{t\Sigma_2})^\wedge_a)=(\pi_*(X^{t\Sigma_2}))^\wedge_a=(\pi_*X((x))/[2]_F(x))^\wedge_a.$$
By the Lemma \ref{lemma:algebra}, we have
$$f_* \colon (\pi_*X((x))/[2]_F(x))^\wedge_a \longrightarrow 
(\pi_*X)^\wedge_a((x))/[2]_F(x)$$
is an isomorphism. This completes the proof.
\end{proof}

\subsection{The parametrized Tate construction}\label{SS:ParamTate}

We now recall some facts about the parametrized Tate construction. Let $(G,\sigma_G)$ be a Real Lie group. Then $\sigma_G$ is the image of the generator $\sigma \in C_2$ under some group homomorphism $\tau : C_2 \to Aut(G)$, and $\tau$ defines a semidirect product $G \rtimes C_2$. 

\begin{defin}
Let $\cf_\tau$ denote the family of all closed subgroups $H \subseteq G \rtimes C_2$ satisfying $H \cap G = \{e\}$. Let $E_{C_2}^\tau G := E\cf_\tau$ denote the unique $G \rtimes C_2$-equivariant homotopy type satisfying
$$(E_{C_2}^\tau G)^H \simeq \begin{cases} * \quad & \text{ if } H \in \cf_\tau, \\
\emptyset \quad & \text{ if } H \notin \cf_\tau.
\end{cases}
$$
Let $B_{C_2}^\tau G := (E_{C_2}^\tau G) / G.$ We say that $B_{C_2}^\tau G$ is the \emph{$C_2$-equivariant classifying space of $G$}. 
\end{defin}

\begin{conv}
Note that the $G \rtimes C_2$-equivariant homotopy type of $E^\tau_{C_2}G$ depends on a choice of group homomorphism $\tau: C_2 \to Aut(G)$, so we ought to always write $E_{C_2}^\tau G$. However, we are primarily interested in the case where $G$ is an abelian group, in which case we always take $\tau : C_2 \to Aut(G)$ to be the inversion involution. We will therefore suppress $\tau$ from the notation unless we specify a different involution. 
\end{conv}

\begin{exm}
There is an equivalence of $O(2)$-spaces
$$E_{C_2}S^1 \simeq S(\c^{\oplus \infty}),$$
where the $O(2) = S^1 \rtimes C_2$-action on the right-hand side is determined by letting $C_2$ act on $\c$ by complex conjugation and $S^1$ act by rotation. Consequently, there is an equivalence of $C_2$-spaces
$$B_{C_2}S^1 \simeq \c\p^\infty,$$
where $C_2$ acts on the right-hand side by complex conjugation. Compare with \cite[Def. 2.1]{HK01}. 

Similarly, there is an equivalence of $D_4$-spaces
$$E_{C_2}\Sigma_2 \simeq S(\c^{\oplus \infty}),$$
where the $D_4 = \Sigma_2 \times C_2$-action on the right-hand side is obtained by restriction of the $O(2)$-action above. This induces an equivalence of $C_2$-spaces
$$B_{C_2}\Sigma_2 \simeq \r\p^\infty,$$
where $\r\p^\infty$ is the space of real lines in $\c^{\oplus \infty}$ and $C_2$ acts by complex conjugation. Compare with \cite[Sec. 3]{KW08}. 
\end{exm}

We make the following definition in view of several results of the third author with Shah: \cite[Def. 1.6]{QS21a}, \cite[Obs. 1.9]{QS21a}, \cite[Obs. 1.10]{QS21a}, and \cite[Thm. C]{QS21a}. Alternatively, the following definition specializes \cite[Sec. 17]{GM95} to the family of graph subgroups of $C_2$ in $G$, cf. \cite[Rmk. 3.17]{QS21a}.

\begin{defin}\label{def:parametrizedTate}
Let $X$ be a $C_2$-spectrum with $G$-action. The \emph{parametrized $G$-homotopy orbits} of $X$ are defined by
$$X_{h_{C_2}G} := (F(E_{C_2}G_+,X) \wedge E_{C_2}G_+)^G$$
and the \emph{parametrized $G$-homotopy fixed points} of $X$ are defined by
$$X^{h_{C_2}G} := F(E_{C_2}G_+, X)^G.$$
There is a cofiber sequence
$$E_{C_2}G_+ \to S^0 \to \widetilde{E_{C_2}G}$$
which gives rise to a cofiber sequence of spectra
$$X_{h_{C_2}G} \to X^{h_{C_2}G} \to (F(E_{C_2}G_+,X) \wedge \widetilde{E_{C_2}G})^G.$$
We define the cofiber to be the \emph{parametrized $G$-Tate construction} $X^{t_{C_2}G}$. 
\end{defin}

When the extension $G \rtimes C_2$ determined by $\tau : C_2 \to Aut(G)$ satisfies a certain representation-theoretic hypothesis, the parametrized $G$-Tate construction for a $C_2$-spectrum equipped with a trivial $G$-action may be identified with an inverse limit. We obtain the following as a specialization of \cite[Thm. B]{QS21a}, which generalizes \cite[Thm. 16.1]{GM95} to families:

\begin{thm}\label{Thm:ThomGeneral}\cite[Thm. B]{QS21a}
Let $\tau: C_2 \to Aut(G)$ and suppose that there exists a $G \rtimes C_2$-representation $V$ with $V^H = 0$ for $H \notin \cf_\tau$ and $V^H \neq 0$ for $H \in \cf_\tau$. Let $U$ be a complete $G \rtimes C_2$-universe and let $i : U^{G} \to U$ be the inclusion of the $G$-fixed universe. Let $X$ be a $C_2$-spectrum. There is an equivalence of $C_2$-spectra
$$(i_*X)^{t_{C_2}G} \simeq \limt{n}(Th((-nV) \to B_{C_2}^{\tau}G) \wedge \Sigma X),$$
where $i_*(-)$ is the functor which associates a genuine $G \rtimes C_2$-spectrum to an $\cf(C_2,G)$-spectrum indexed on the complete $C_2$-universe $U^G$. 
\end{thm}

\begin{defin}\label{Exm:Qin}\cite[Def. A.19]{Qui21a}
When $G = \Sigma_2$, we will let $Q^\infty_{-n} := Th(-n\gamma \to B_{C_2}\Sigma_2)$ denote the $C_2$-equivariant Thom spectrum appearing in the previous theorem. 
\end{defin}

In particular, we obtain the following formula for the parametrized Tate construction as an inverse limit:

\begin{cor}\label{Cor:TLim}
Let $U$ be a complete $\Sigma_2 \rtimes C_2$-universe and let $i : U^{\Sigma_2} \to U$ be the inclusion of the $\Sigma_2$-fixed universe. Let $X$ be a $C_2$-spectrum. Then there is an equivalence of $C_2$-spectra
$$(i_*X)^{t_{C_2}\Sigma_2} \simeq \limt{i}(Q^\infty_{-i} \wedge \Sigma X).$$
\end{cor}

\subsection{The parametrized $\Sigma_2$-Tate construction}\label{Sec:ParamSigma2}


The following discussion adapts some results from \cite[Sec. 3]{GM95} to the $C_2$-equivariant setting; many similar ideas appear in \cite{QS21a}. 

\begin{defin}
Let $X$ be a $\Sigma_2 \rtimes C_2$-spectrum. We define the \emph{$\Sigma_2$-free $\Sigma_2 \rtimes C_2$-spectrum} associated to $X$ by
$$f_p(X) := X \wedge E_{C_2} {\Sigma_2}_+.$$
The \emph{parametrized geometric completion} of $X$ is defined by
$$c_p(X) := F(E_{C_2}{\Sigma_2}_+,X).$$
The \emph{parametrized $\Sigma_2$-Tate $\Sigma_2 \rtimes C_2$-spectrum} associated to $X$ is defined by 
$$t_p(X) := F(E_{C_2}{\Sigma_2}_+, X) \wedge \widetilde{E_{C_2}{\Sigma_2}}.$$
\end{defin}

The diagonal map of $E_{C_2} {\Sigma_2}_+$ and the $C_2$-equivariant equivalences
$$E_{C_2} {\Sigma_2}_+ \wedge E_{C_2} {\Sigma_2}_+ \simeq E_{C_2} {\Sigma_2}_+ \quad \text{ and } \quad \widetilde{E_{C_2} \Sigma_2} \wedge \widetilde{E_{C_2}\Sigma_2} \simeq \widetilde{E_{C_2} \Sigma_2}$$
give rise to a commutative diagram of associative and commutative natural pairings, called the \emph{parametrized norm pairing diagram}:
\[
\begin{tikzcd}
f_p(X) \wedge f_p(X') \arrow{r} \arrow{d} & c_p(X) \wedge c_p(X') \arrow{r} \arrow{d} & t_p(X) \wedge t_p(X') \arrow{d} \\
f_p(X \wedge X') \arrow{r} & c_p(X \wedge X') \arrow{r} & t_p(X \wedge X').
\end{tikzcd}
\]
This implies the following analog of \cite[Prop. 3.5]{GM95}. 

\begin{prop}\label{PropGM35}
If $X$ is a ring $\Sigma_2 \rtimes C_2$-spectrum, then $c_p(X)$ and $t_p(X)$ are ring $\Sigma_2 \rtimes C_2$-spectra and the following parametrized of the  norm-restriction diagram (cf. \cite[Eqn. D, Sec. 17]{GM95}) is a commutative diagram of ring $\Sigma_2 \rtimes C_2$-spectra:
\[
\begin{tikzcd}
X \arrow{r} \arrow{d} & X \wedge \widetilde{E_{C_2} \Sigma_2} \arrow{d} \\
c_p(X) \arrow{r} & t_p(X).
\end{tikzcd}
\]
If $X$ is commutative, then so are $c_p(X)$ and $t_p(X)$. If $M$ is an $X$-module $\Sigma_2 \rtimes C_2$-spectrum, then $c_p(M)$ is a $c_p(X)$-module $\Sigma_2 \rtimes C_2$-spectrum and $t_p(M)$ is a $t_p(M)$-module $\Sigma_2 \rtimes C_2$-spectrum. 
\end{prop}

By \emph{ring $\Sigma_2 \rtimes C_2$-spectrum}, we mean a $\Sigma_2 \rtimes C_2$-spectrum with unit and multiplication for which the usual monoid diagrams hold in the homotopy category.





\subsection{Comparison between classical and parametrized Tate constructions}\label{SS:Compare}
Let $\e$ be a $C_2$-spectrum. If $U \cong \c^\infty$ is a complete $\Sigma_2 \rtimes C_2$-universe (where $C_2$ acts by complex conjugation), the inclusion of fixed points $U^{C_2} \cong \r^\infty \hookrightarrow \c^\infty \cong U$ induces a map of universal spaces 
$$E\Sigma_2 \longrightarrow E_{C_2}\Sigma_2.$$ 
We may view this as a $\Sigma_2 \rtimes C_2$-equivariant map with trivial $C_2$ action on the source. This map gives a $\Sigma_2$-equivariant equivalence upon forgetting the $C_2$-action. Taking orbits under the $\Sigma_2$-action, we get a $C_2$-equivariant map of classifying spaces (with trivial $C_2$-action on the source)
$$B\Sigma_2 \longrightarrow B_{C_2}\Sigma_2$$
which is an underlying nonequivariant equivalence.

We obtain some useful equivalences if we assume that the $C_2$-spectrum $\e$ is either free or cofree.

\begin{lem}\label{tatecomparison}
\begin{enumerate}
\item If $\e$ is $C_2$-free, then we have $C_2$-equivariant equivalences
$${E\Sigma_2}_+ \wedge \e \longrightarrow {E_{C_2}\Sigma_2}_+ \wedge \e,$$
$${B\Sigma_2}_+ \wedge \e \longrightarrow {B_{C_2}\Sigma_2}_+ \wedge \e.$$
\item If $\e$ is $C_2$-cofree, then we have $C_2$-equivariant equivalences
$$F({E_{C_2}\Sigma_2}_+, \e) \longrightarrow F({E\Sigma_2}_+, \e),$$
$$F({B_{C_2}\Sigma_2}_+, \e) \longrightarrow F({B\Sigma_2}_+, \e).$$
\end{enumerate}
\end{lem}

\begin{proof} This is an immediate consequence of Lemma \ref{freecofreecomparison} and the fact that the (equivariant) inclusions $B\Sigma_2 \rightarrow B_{C_2}\Sigma_2$ and $E\Sigma_2 \rightarrow E_{C_2}\Sigma_2$ are underlying nonequivariant equivalences.
\end{proof}

We have the following commutative square of $\Sigma_2 \rtimes C_2$-spectra which relates the ordinary and parametrized Tate constructions:
$$\xymatrix{\widetilde{E\Sigma_2} \wedge F({E\Sigma_2}_+, \e) \ar@{->}[r] & \widetilde{E_{C_2}\Sigma_2} \wedge F({E\Sigma_2}_+, \e) \\
\widetilde{E\Sigma_2} \wedge F({E_{C_2}\Sigma_2}_+, \e) \ar@{->}[r] \ar@{->}[u] & \widetilde{E_{C_2}\Sigma_2} \wedge F({E_{C_2}\Sigma_2}_+, \e).\ar@{->}[u]}$$
Applying $\Sigma_2$-fixed points to each corner, the left-top corner becomes $\e^{t\Sigma_2}$ and the bottom-right corner becomes $\e^{t_{C_2}\Sigma_2}$.  We therefore have a zigzag relating $\e^{t\Sigma_2}$ and $\e^{t_{C_2}\Sigma_2}$. 

In the case that $\e$ is $C_2$-cofree, the vertical maps are $C_2$-equivariant equivalences. However, to obtain a map $\e^{t\Sigma_2} \to \e^{t_{C_2}\Sigma_2}$, one of the vertical maps must actually be a $\Sigma_2 \rtimes C_2$-equivariant equivalence. 

\begin{lem}\label{tatecomparison2}
If $\e$ is $C_2$-cofree with trivial $\Sigma_2$-action, then there is a $\Sigma_2 \rtimes C_2$-equivalence
$$F({E_{C_2}\Sigma_2}_+, \e) \to F({E\Sigma_2}_+, \e).$$
\end{lem}

\begin{proof}
The map is a $C_2$-equivalence by Lemma \ref{tatecomparison} since $\e$ is $C_2$-cofree. It remains to check that the map is an equivalence after taking fixed points with respect to $\Delta$, $\Sigma_2$, and $\Sigma_2 \rtimes C_2$ (where $\Delta$ is the diagonal subgroup of $\Sigma_2 \rtimes C_2$). 

We first check that
$$F({E_{C_2}\Sigma_2}_+,\e)^\Delta \to F({E\Sigma_2}_+,\e)^\Delta$$
is an equivalence. We have a sequence of $\Sigma_2 \rtimes C_2$-equivalences
$$F({E_{C_2}\Sigma_2}_+,\e)^\Delta \simeq F({E_{C_2}\Sigma_2}_+, F({EC_2}_+,\e))^\Delta \simeq F((E_{C_2}\Sigma_2 \times EC_2)_+, \e)^\Delta$$
which follow from adjunction along with the fact that $F({EC_2}_+,\e) \to \e$ is a $\Sigma_2 \rtimes C_2$-equivalence if $\e$ is cofree with trivial $\Sigma_2$-action. Now, we observe that the projection
$$E_{C_2}\Sigma_2 \times E\Sigma_2 \to E\Sigma_2$$
is an $\cf(C_2,\Sigma_2)$-equivalence of $\Sigma_2 \rtimes C_2$-spaces. Therefore the source and target are equivalent after taking $\Delta$ fixed points. 

A similar argument works for fixed points with respect to $\Sigma_2$ and $\Sigma_2 \rtimes C_2$. We may replace $F({E\Sigma_2}_+, \e)$ by $F((E\Sigma_2 \times EC_2)_+, \e)$, and the projection map $E\Sigma_2 \times EC_2 \to E\Sigma_2$ is an equivalence of spaces on $\Sigma_2$- and $\Sigma_2 \rtimes C_2$-fixed points. 
\end{proof}

Therefore when $\e$ is $C_2$-cofree with trivial $\Sigma_2$-action, the right-hand vertical map is an equivalence on $\Sigma_2$-fixed points and we obtain a map
$$E^{t\Sigma_2} \to E^{t_{C_2}\Sigma_2}.$$
This map is not generally an equivalence, but in the next section, we will show that it is an equivalence under further orientability assumptions for $\e$.

\section{Blueshift for Real oriented cohomology theories}\label{blueshift}

Our first main results, blueshift for the Tate constructions of Real Johnson--Wilson theories, are proven in this section. In Section \ref{SS:ClassicalBlue}, we summarize the proof of a splitting for $L_{K(n-1)}(E(n)^{t\Sigma_2})$ given by Ando, Morava, and Sadofsky in \cite{AMS98}. We modify their proof in Section \ref{C2blueshift} to prove splittings of the parametrized Tate construction of Real Johnson--Wilson theories and the ordinary Tate construction of their fixed points. 

\subsection{The classical Tate construction of $E(n)$}\label{SS:ClassicalBlue}
In \cite{AMS98}, Ando, Morava, and Sadofsky analyze the $K(n-1)$-localization of the Tate construction $E(n)^{tC_p}$ using complex orientations and formal group laws. Recall that if $E$ is a complex oriented cohomology theory, then there is an isomorphism
$$E^*(\c P^\infty) \cong E^*[[x]]$$
where $|x| = 2$. The map $\c P^\infty \times \c P^\infty \to \c P^\infty$ which classifies the tensor product of complex line bundles induces a map
$$E^*[[z]] \cong E^*(\c P^\infty) \to E^*(\c P^\infty) \otimes E^*(\c P^\infty) \cong E^*[[x,y]]$$
which determines a formal group law over $E^*$. There is then an isomorphism
$$E^*(B\Sigma_2) \cong E^*[[x]]/([2](x))$$
where $[2](x)$ is the $2$-series of the formal group law associated to $E$. 

The starting point of the work of Ando--Morava--Sadofsky was the following identification.

\begin{lem}\cite[Lem. 2.1]{AMS98}
Suppose that $E$ is a complex oriented spectrum and that its $2$-series $[2](x)$ is not a $0$ divisor in $E^*[[x]].$ Then there is an isomorphism of rings
$$\pi_{-*}(E^{t\Sigma_2}) \cong E^*((x))/([2](x)).$$
\end{lem}

This allows them to analyze $E^{t\Sigma_2}$ using more algebraic methods. Using formal group law calculations, they prove the following. 

\begin{prop}\cite[Prop. 2.11]{AMS98}
There is an isomorphism
$$(\pi_*E(n)^{t\Sigma_2})^\wedge_{I_{n-1}} \to E(n-1)_*((x))^\wedge_{I_{n-1}}$$
where $|x| = -2$. 
\end{prop}
\begin{rem2}
We use $E^*=E_{-*}$ and change the degree of $x$ from $2$ to $-2$ to write
$$\pi_{-*}(E^{t\Sigma_2}) \cong E^*((x))/([2](x))$$
as
$$\pi_{*}(E^{t\Sigma_2}) \cong E_*((x))/([2](x)).$$
\end{rem2}

They lift this isomorphism to an equivalence of $K(n-1)$-local spectra in \cite[Sec. 3]{AMS98}. This equivalence relies on several ideas from chromatic homotopy theory, including explicit formulas for Bousfield localization with respect to $K(n)$ and certain $MU$-module structures. It also relies on understanding the relationship between the $I_{n-1}$-adic filtration of $\pi_*(E(n)^{t\Sigma_2})$ and the ``$x$-adic" filtration which arises from the identification $E^{t\Sigma_2} \simeq \limt{n} (E \wedge \Sigma RP^\infty_{-n})$ and the cellular filtration of $RP^\infty_{-n}$. In the end, they prove

\begin{thm}\cite[Thm. 3.10]{AMS98}
There is a map of spectra
$$\limt{i} \bigvee_{j \leq -i} \Sigma^{2i}E(n-1) \to (E(n)^{t\Sigma_2})^\wedge_{I_{n-1}}$$
that becomes an isomorphism on homotopy groups after completion at $I_{n-1}$, or equivalently after Bousfield localization with respect to $K(n-1)$.
\end{thm}

\subsection{The  ordinary and parametrized $\Sigma_2$-Tate constructions for $\e(n)$}\label{Sec:RealSplitting}
The Real Johnson--Wilson theory $\e(n)$ is free and cofree, and so by Lemma \ref{tatecomparison}, we have a comparison map between the ordinary and parametrized  $\Sigma_2$-Tate constructions
$$\e(n)^{t\Sigma_2} \longrightarrow \e(n)^{t_{C_2}\Sigma_2}$$
In this section, we will show this map is an equivalence. Along the way, we will calculate the homotopy groups of both sides. 

We begin with the observation that the values of $\e(n)$-cohomology on the parametrized and ordinary classifying spaces $B_{C_2}\Sigma_2$ and $B\Sigma_2$ agree by Lemma \ref{tatecomparison}, since $\e(n)$ is cofree. These are given by the following Lemma from \cite{KW08}:\footnote{Kitchloo and Wilson prove the lemma for $\e = \e(n)$, but their proof works equally well for any Real oriented cohomology theory in which $[2](x)$ is not a zero divisor.}


\begin{lem}\cite[Lemma 3.1]{KW08}
Let $\e$ be a Real oriented cohomology theory in which $[2](x)$ is not a zero divisor. Then there is an $RO(C_2)$-graded isomorphism
$$\e^\star(B_{C_2}\Sigma_2) \cong \e^\star[[x]]/([2](x))$$
where $|x| = \rho$ and $[2](x)$ is the $2$-series of the formal group law associated to $E$. 
\end{lem}

The two Tate constructions $\e(n)^{t\Sigma_2}$ and $\e(n)^{t_{C_2}\Sigma_2}$ are homotopy limits of $\e(n)$ smashed with Thom spectra of bundles over $B\Sigma_2$ and $B_{C_2}\Sigma_2$, respectively. We now describe them more explicitly:

\begin{enumerate}
\item Let $\xi_V$ be the line bundle over $B\Sigma_2$ associated to the $\Sigma_2$ sign representation, $V$. Note that $\Sigma_2$ acts freely on the unit sphere in this representation. The total space of $\xi_V$ is given by $E\Sigma_2 \times_{\Sigma_2} V$.
\item Let $\xi_W$ be the bundle over $B_{C_2}\Sigma_2$ associated to the $\Sigma_2 \rtimes C_2$ representation $W:=\mathbb{C}$, where $C_2$ acts by complex conjugation and $\Sigma_2$ acts by rotation. The unit sphere $S(W)$ is a $\mathcal{F}(C_2, \Sigma_2)$ space.
\end{enumerate}

If we forget the $C_2$-action, then $W$ is just $2V$ as a $\Sigma_2$-representation. While $C_2$ does not act on $B\Sigma_2$, it still acts fiberwise on the pullback of $\xi_W$ along $B\Sigma_2 \rightarrow B_{C_2}\Sigma_2$. Using the inclusion of $\xi_V$ as the $C_2$-fixed points of the pullback of $\xi_W$, we have maps of Thom spectra
$$B\Sigma_2^{k\xi_V} \rightarrow B_{C_2}\Sigma_2^{k\xi_W}$$
which are $C_2$-equivariant, provided we take trivial $C_2$-action on the source.


We can compute the homotopy groups of each Tate construction by rewriting them as colimits. We start with the ordinary Tate construction. Note that $\widetilde{E\Sigma_2} =\hocolim_n S^{nV}$, and let $c(\e):=F({E\Sigma_2}_+, i_*\e)$. 




\begin{thm}\label{Thm:TateCoeffsE}
There is an isomorphism
$$\pi_{-*}(\e(n)^{t\Sigma_2}) \cong \e(n)^{\star}((\widehat{u}))/([2](\widehat{u}))$$
where $|\widehat{u}| = 1-\lambda$.
\end{thm}

\begin{proof} 
By the definition of the Tate construction, we have that 
$$\pi_{\star}(\e(n)^{t\Sigma_2}) = \pi_\star (c(\e(n)) \wedge \widetilde{E\Sigma_2})^{\Sigma_2}.$$
We identify the right hand as $\colim_k \pi_\star (c(\e(n) \wedge S^{kV})^{\Sigma_2}$.
Since $i_*\e(n)$ is a split $\Sigma_2$-spectrum and ${E\Sigma_2}_+$ is $\Sigma_2$-free, there is an isomorphism
$$\pi_\star (c(i_*\e(n)) \wedge S^{kV})^{\Sigma_2} \cong \pi_\star F(S^{-k V} \wedge_{\Sigma_2} {E\Sigma_2}_+,i_*\e(n)).$$
Now $S^{-kV} \wedge_{\Sigma_2} {E\Sigma_2}_+$ is exactly the Thom spectrum $B\Sigma_2^{-k\xi_V}$, so we get
$$\pi_{-\star}(\e(n)^{t\Sigma_2}) \cong \colim \e(n)^\star(B\Sigma_2^{-k\xi_V}).$$
Now we identify the right-hand side. Recall from Theorem \ref{Thm:Or} that vector bundles that are multiples of $2^{n+1}$ are $\e(n)$-orientable. We choose a cofinal subsequence in the above colimit given by $k=2^{n+2}j$ (note that we give ourselves an extra factor of 2 to work with). Then there is a Thom isomorphism
$$\e(n)^*(B\Sigma_2^{-2^{n+2}j\xi_V}) \cong \e(n)^{*+2^{n+2}j}(B\Sigma_2)$$
Now we need to identify the maps in this colimit. The map $B\Sigma_2^{-2^{n+2}(j+1)\xi_V} \longrightarrow B\Sigma_2^{-2^{n+2}j\xi_V}$ induces a map
$$\e(n)^{*+2^{n+2}j}(B\Sigma_2) \longrightarrow \e(n)^{*+2^{n+2}(j+1)}(B\Sigma_2)$$
which raises degree by $2^{n+2}$ and is given by multiplication by the Euler class of $2^{n+2}\xi_V$. It remains to identify this Euler class. 

Recall that in integer degrees, we have $\e(n)^{\star}(B\Sigma_2)=\e(n)^\star[[\widehat{u}]]/([2](\widehat{u}))$, where $\widehat{u}$ is obtained from $\overline{u} \in \e(n)^\rho(B\Sigma_2)$ by multiplying by $y$, and the image of $\widehat{u}$ in (ordinary, nonequivariant) $E(n)$-cohomology is $v_n^{2^n-1}u$, where $u$ is the first Chern class of the tautological (complex) line bundle in degree 2. By Theorem \ref{Thm:Or}, we have that the Euler class of $2^{n+2}\xi_V$ in $ER(n)$-cohomology maps under the forgetful map to a unit power series multiple of the $E(n)$-Euler class of $2^{n+2}\xi_V$, which is given by $u^{2^{n+1}}$. 

We claim that $ER(n)^*(B\Sigma_2)$ injects into $E(n)^*(B\Sigma_2)$ in degrees that are multiples of $2^{n+1}$. Since $ER(n)^*(B\Sigma_2)$ is multiplicatively generated by a single class $\widehat{u}$ in degree divisible by $2^{n+1}$, the claim follows from the corresponding fact about the coefficients. The kernel of the map from $ER(n)^*$ to $E(n)^*$ is generated by the class $x$, and injectivity in degrees $2^{n+1}$ follows from the description of the coefficients of $ER(n)$, e.g. in \cite[Theorem 3.1]{KW15}

Thus, we identify the $ER(n)$ Euler class with a unit multiple of $v_n^{-2^{n+1}(2^n-1)}\widehat{u}^{2^{n+1}}$ since the map to $E(n)$-cohomology sends this to the $E(n)$-Euler class and we are in degrees in which the map from $ER(n)$ cohomology to $E(n)$-cohomology is injective.\footnote{Choosing the cofinal subsequence $k=2^{n+1}j$ at the beginning of the proof would have still given us Thom isomorphisms, but the $ER(n)$-Euler class would not  have the nice form that it does when we give ourselves an extra factor of 2 to work with.}

It follows that the maps in the direct limit are multiplication by a unit multiple of $\widehat{u}^{2^{n+1}}$. This has the effect of inverting $\widehat{u}$ in the colimit and this completes the computation.
\end{proof}


We now calculate the homotopy groups of the parametrized Tate construction. By mimicking the above calculations for the bundle $\xi_W$ over $B_{C_2}\Sigma_2$, we have an isomorphism
$$\pi_{-\star}\e(n)^{t_{C_2}\Sigma_2} \cong \colim_k \e(n)^\star(B_{C_2}\Sigma_2^{-k\xi_W}).$$
As in the ordinary Tate construction, we must identify the right-hand side. More generally, we have the following calculation for any Real oriented cohomology theory. 

\begin{thm}\label{Thm:ParamTateCoeffs}
Let $\e$ be a Real oriented cohomology theory. Then there is an $RO(C_2)$-graded isomorphism
$$\pi_{-\star}\e^{t_{C_2}\Sigma_2} \cong \e^{\star}((\overline{u}))/([2](\overline{u}))$$
where $|\overline{u}| = \rho$. 
\end{thm}

\begin{proof}
Recall that $\xi_W$ has twice the (real) dimension of $\xi_V$. The bundle $\xi_W$  itself (not just a $2^{n}$ multiple of it) is orientable with respect to $\e$, so we have a Thom isomorphism
$$\e^\star(B_{C_2}\Sigma_2^{-k\xi_W}) \cong \e^{\star+k(1+\sigma)}(B_{C_2}\Sigma_2)$$
The map $B_{C_2}\Sigma_2^{-(k+1)\xi_W} \rightarrow B_{C_2}\Sigma_2^{-k\xi_W}$ induces in $\e$-cohomology the map
$$\e^{\star+k(1+\sigma)}(B_{C_2}\Sigma_2) \longrightarrow \e^{\star+(k+1)(1+\sigma)}(B_{C_2}\Sigma_2)$$
which is multiplication by the Euler class of $\xi_W$, given by $\overline{u} \in \e^{1+\sigma}(B_{C_2}\Sigma_2)$. (This is the first Chern class of the tautological Real bundle over $RP^\infty$, which pulls back from the Real bundle over $\mathbb{CP}^\infty$).

Since we know that
$$\e^{\star}(B_{C_2}\Sigma_2) \cong \e^{\star}[[\overline{u}]]/([2](\overline{u}))$$
it follows that in the colimit, the class $\overline{u}$ is inverted and we get
$$\pi_{-\star}\e^{t_{C_2}\Sigma_2} \cong \e^{\star}((\overline{u}))/([2](\overline{u}))$$
\end{proof}

We conclude this section by comparing the two Tate constructions. 

\begin{thm}\label{Thm:Agree}
The map
$$\e(n)^{t\Sigma_2} \longrightarrow \e(n)^{t_{C_2}\Sigma_2}$$
is a $C_2$-equivariant weak equivalence.
\end{thm}

\begin{proof}
Going back to the cofinal subsequence $k=2^{n+2}j$ for both $\xi_V$ and $\xi_W$, this boils down to identifying the map
\begin{align*}
\e(n)^\star(B\Sigma_2) \cong \e(n)^{\star-2^{n+2}j}(B\Sigma_2^{-2^{n+2}j\xi_V}) \longrightarrow & \e(n)^{\star-2^{n+2}j}(B_{C_2}\Sigma_2^{-2^{n+2}j\xi_W} ) \\
& \cong \e(n)^{\star+2^{n+2}j\sigma}(B_{C_2}\Sigma_2) \cong \e(n)^{\star-2^{n+2}j\lambda}(B_{C_2}\Sigma_2)
\end{align*}
where the last map is given by multiplying by $y^{2^{n+2}j}$ to shift into integer degrees. This is a map of $\e(n)^\star(B_{C_2}\Sigma_2)=\e(n)^\star(B\Sigma_2)$-modules, and so it suffices to calculate the image of $1$. In integer degrees that are multiples of $2^{n+2}$, we have that $\e(n)^\star(B\Sigma_2)\cong \e(n)^\star(B_{C_2}\Sigma_2)$ injects under the forgetful map into $E(n)^*(B\Sigma_2)\cong E(n)^*(B_{C_2}\Sigma_2)$. Since the target of $1 \in \e(n)^0(B\Sigma_2)$ will be in $\e(n)^{-2^{n+1}\lambda}(B_{C_2}\Sigma_2)$, we may calculate its image by forgetting to $E(n)$-cohomology.

Nonequivariantly, the corresponding composite is given by
\begin{align*}
E(n)^*(B\Sigma_2) \cong E(n)^{*-2^{n+2}j}(B\Sigma_2^{-2^{n+2}j\xi_V}) \longrightarrow & E(n)^{*-2^{n+2}j}(B_{C_2}\Sigma_2^{-2^{n+2}j\xi_W}) \\
&  \cong E(n)^{*+2^{n+2}j}(B_{C_2}\Sigma_2) \cong E(n)^{*-2^{n+2}j\lambda}(B_{C_2}\Sigma_2)
\end{align*}
where the last map is multiplication by $v_n^{2^{n+2}j(2^n-1)}$ (which underlies multiplication by $y^{2^{n+2}j}$).

Non-equivariantly, the map $B\Sigma_2 \longrightarrow B_{C_2}\Sigma_2$ is an equivalence, and the bundle $\xi_W$ may be identified with $2\xi_V$. It follows that the image of 1 under the above composite is given by the Euler class of $2^{n+2}j\xi_V$. We have already seen that the underlying nonequivariant Euler class of $2^{n+2}j\xi_V$ is given by a unit multiple of $u^{2^{n+1}j}$.

It follows that equivariantly, the map we are interested in sends $1$ to a unit multiple of $\overline{u}^{2^{n+1}}$. Thus, from the computation of both Tate constructions, it follows that in the inverse limit, this induces an isomorphism.
\end{proof}

\subsection{$C_2$-equivariant blueshift for $\e(n)$}\label{C2blueshift}

Now that we know the ordinary and parametrized Tate constructions of $\e(n)$ agree and we know their homotopy groups, we want to identify the homotopy type of the Tate construction after a certain completion. We begin by proving a splitting in completed $C_2$-equivariant homotopy groups, first for $BP\r \langle n \rangle$, and then for $\e(n)$. Since the analogous computations in \cite[Sec. 2]{AMS98} only depend on formal group law calculations, they carry over easily to the $C_2$-equivariant setting. In particular, all of the computations there are valid after replacing $v_k$ by $\bar{v}_k$. 

\begin{prop}\label{C2eqvtBPRsplit}
There is an isomorphism
\begin{align*}
\pi^{C_2}_\star(BP\r \langle n \rangle^{t_{C_2}\Sigma_2}) & \cong BP\r \langle n \rangle_\star ((x)) /([p](x)) \\
& \cong \prod_{k \in \z} \Sigma^{\rho k} BP \r \langle n-1 \rangle^\wedge_{p,\star} \quad &(*)\\
& \cong BP \r \langle n-1 \rangle^\wedge_{p,\star} ((x)).
\end{align*}
\end{prop}

\begin{proof}
We begin with the identification $\pi^{C_2}_\star(BP\r \langle n \rangle^{t_{C_2}\Sigma_2}) \cong BP\r \langle n \rangle^{\star}((\overline{u}))/([2](\overline{u}))$ from Theorem \ref{Thm:ParamTateCoeffs}. The $RO(C_2)$-graded homotopy groups $BP\r\langle n \rangle_\star$ can be found in \cite{Hu02}. The formal group law over $BP \r \langle n \rangle_{*\rho}$ associated to $BP \r \langle n \rangle$ is isomorphic to the formal group law over $BP \langle n \rangle_*$ associated to $BP \langle n \rangle$, with isomorphism given by $\bar{v}_k \mapsto v_k$. The proof of \cite[Prop. 2.3]{AMS98} depends only on computations with the formal group law over $BP \langle n \rangle_*$ associated to $BP \langle n \rangle$, so the proof carries over to the $C_2$-equivariant setting to give an isomorphism $\pi^{C_2}_{*\rho}(BP\r \langle n \rangle^{t_{C_2}\Sigma_2}) \simeq BP\r \langle n-1 \rangle^\wedge_{p, *\rho} ((x))$. We note that under this isomorphism, $x \mapsto x$, $\bar{v}_k \mapsto \bar{v}_k$ for $k < n$, and $\bar{v}_n$ maps to a certain power series in $x$ over $BP\r\langle n-1 \rangle_{*\rho}$. 

We can extend the isomorphism to $RO(C_2)$-grading by sending $u_{2\sigma} \mapsto u_{2\sigma}$ and $a_\sigma \mapsto a_\sigma$. Indeed, surjectivity is clear. To see the map is injective, we just need to check that the image of $\bar{v}_n$ in $BP\r\langle n-1 \rangle_\star$ is $a_\sigma^{2^n-1}$-torsion. This follows from the fact that the image is a power series with coefficients in $\bar{I}_n$, and every element of $\bar{I}_n$ is $a_\sigma^{2^n-1}$-torsion. 
\end{proof}

The next proposition follows from a similar modification of the proof of \cite[Prop. 2.11]{AMS98}.

\begin{prop}\label{AlgTateIso}
The map $(*)$ above extends to an isomorphism
$$\e(n)_\star ((x))/([2](x))^\wedge_{\bar{I}_{n-1}} \cong \e(n-1)_\star ((x))^\wedge_{\bar{I}_{n-1}}$$
where
$$\pi^{C_2}_\star((\bar{v}_n^{-1} BP \r \langle n \rangle)^{t_{C_2}\Sigma_2})^\wedge_{\bar{I}_{n-1}} \cong \pi^{C_2}_\star(\e(n)^{t_{C_2}\Sigma_2})^\wedge_{\bar{I}_{n-1}} \cong \e(n)_\star ((x))/([2](x))^{\wedge}_{\bar{I}_{n-1}}$$
and 
$$\e(n-1)_\star ((x))^\wedge_{\bar{I}_{n-1}} \cong (\bar{v}^{-1}_{n-1} BP\r \langle n -1 \rangle)((x))^\wedge_{\bar{I}_{n-1}}.$$
\end{prop}

We now turn to proving the spectrum-level result, which is the main result of this section. 

\begin{thm}\label{C2eqvtERsplit}
There is a map of spectra
$$\limt{i} \bigvee_{j \leq i} \Sigma^{\rho j} \e(n-1) \to (\e(n)^{t_{C_2}\Sigma_2})^\wedge_{\bar{I}_{n-1}}$$
that becomes an isomorphism on homotopy groups after completion at $\bar{I}_{n-1}$ or equivalently after localization at $\k(n-1)$.
\end{thm}

Our argument follows \cite[Sec. 3]{AMS98}, with the caveat that we must make sense of certain constructions $C_2$-equivariantly. In Section \ref{Sec:Completion}, we described $X^\wedge_{\overline{I}_{n-1}}$ for an $\mathbb{MR}(n)$-module and showed that $\pi_\star^{C_2}(X^\wedge_{\overline{I}_{n-1}}) \cong (\pi_\star^{C_2}X)^\wedge_{\overline{I}_{n-1}}$. In Proposition \ref{modulestructure}, we showed that $(\e (n)^{t _{C_2}\Sigma_2})^\wedge_{\overline{I}_{n-1}}$ has an $\e(n-1)$-module structure. We will use certain generators of $\pi_\star^{C_2}(\e(n)^{t_{C_2}\Sigma_2})^\wedge_{\overline{I}_{n-1}}$ as a $\pi_\star^{C_2}\e(n-1)$-module, then use them to construct a map which will induce an isomorphism in homotopy groups after completion at $\overline{I}_{n-1}$. 






\begin{proof}[Proof of \ref{C2eqvtERsplit}]
We construct the $C_2$-equivariant map whose underlying map is the map in \cite[Thm. 3.10]{AMS98} using a similar approach to Ando--Morava--Sadofsky. Because both side are strongly even and the underlying map is a nonequivariant weak equivalence, this map will be a $C_2$-equivariant weak equivalence by \cite[Lem. 3.4]{HM17}. 

Since $\e(n)$ is a $BP\r$-module and $\bar{v}_{n-1}$ is invertible in $\e(n)^{t_{C_2}\Sigma_2}$, we see that $(\e(n)^{t_{C_2}\Sigma_2})^\wedge_{\overline{I}_{n-1}}$ is a $(\bar{v}_{n-1}^{-1}BP\r)_{\bar{I}_{n-1}}^\wedge$-module. By Proposition \ref{modulestructure}, $(\e(n)^{t_{C_2}\Sigma_2})^\wedge_{\overline{I}_{n-1}}$ is thus an $\e(n-1)$-module. Recall that 
$$\e(n)^{t_{C_2}\Sigma_2} \simeq \limt{i} (Q^\infty_{-i} \wedge \Sigma \e(n)).$$


Now, we take $x^j \in \pi_{j\rho}((\e(n)^{t_{C_2}\Sigma_2})^\wedge_{\overline{I}_{n-1}})$ and use the $\e(n-1)$-structure to construct a sequence of maps
$$\bigvee_{j \leq i} \Sigma^{j\rho} \e(n-1)^\wedge_{\overline{I}_{n-1}} \to (\e(n)^{t_{C_2}\Sigma_2})^\wedge_{\overline{I}_{n-1}}.$$
We then make a map $\mu_{-i}$ by composing this map with the map
$$(\e(n)^{t_{C_2}\Sigma_2})^\wedge_{\overline{I}_{n-1}} \to \left(Q^\infty_{-i-1} \wedge \Sigma \e(n)\right)^\wedge_{\overline{I}_{n-1}}$$
given by Corollary ~\ref{Cor:TLim}.

Taking inverse limits of the maps $\mu_{-i}$ gives a map
$$\limt{i} \left( \bigvee_{ j \leq i} \Sigma^{j\rho} \e(n-1)^\wedge_{\overline{I}_{n-1}}) \right) \overset{f}{\to} \left(\e(n)^{t_{C_2}\Sigma_2}\right)^\wedge_{\overline{I}_{n-1}}.$$
Note that in the target, the parametrized Tate construction has been commuted with $\bar{I}_{n-1}$ completion; this is possible using the same argument as in the proof of Proposition \ref{prop:tatecompletion}. 

Denote the completion of $f$ with respect to $\overline{I}_{n-1}$ by $f^\wedge$. Then the underlying map is the map $f$ in the proof of \cite[Thm. 3.10]{AMS98} after completion, which is a weak equivalence. Because both sides are strongly even, the map is a $C_2$-equivariant equivalence. 

Since the left and right hand sides both satisfy the conditions of Theorem \ref{prop:localizationcomparison}, we have that their $\overline{I}_{n-1}$-completions and their $\k(n-1)$-localizations are equivalent.
\end{proof}

Using Theorem \ref{Thm:Agree} together with a key fact about the coefficients of $\e(n)$, the splitting of Theorem \ref{C2eqvtERsplit} can be extended to a splitting of the classical Tate construction on the fixed points $ER(n):=\e(n)^{C_2}$. 

Recall from Section \ref{sec:coefficients} that we may multiply by an appropriate power of invertible class $y \in \pi_{\lambda + \sigma}\e(n)$ to shift any class in $\pi_\star \e(n)$ into integer degree. In particular, when we do this to the $\bar{v}_i \in \pi_{(2^i-1)\rho}\e(n)$, we get the class $\widehat{v}_i \in \pi_{(2^i-1)(1-\lambda)}ER(n)$. Define the ideal $\widehat{I}_{n-1}:=(2, \widehat{v}_1, \dots, \widehat{v}_{n-2})$. As an ideal of $\pi_\star \e(n)$, it is the same as $\overline{I}_{n-1}=(2, \bar{v}_1, \dots, \bar{v}_{n-2})$.

\begin{thm}\label{ERtatesplitting}\ \\ 
\begin{enumerate}
\item
 For $n \geq 2$ there is a map of spectra
$$\limt{i} \bigvee_{j \leq i} \Sigma^{(1-\lambda_{n-1}) j} ER(n-1) \longrightarrow (ER(n)^{t\Sigma_2})^\wedge_{\widehat{I}_{n-1}}$$
that becomes an isomorphism on homotopy groups after competion at $\widehat{I}_{n-1}$ or equivalently after $K(n-1)$-localization.
\item
For $n=1$, we have an equivalence
$$\limt{i} \bigvee_{j \leq i} \Sigma^{4j} ER(0) \simeq ER(1)^{t\Sigma_2},$$
where $ER(0)\simeq H\mathbb{Q}$ and $ER(1) \simeq KO_{(2)}$.
\end{enumerate}
\end{thm}
\begin{proof} We begin with the case $n \geq 2$. Let $\lambda=\lambda_{n-1}$ and $y=y(n-1) \in \pi_{1-\lambda} \e(n-1)$. We may view $y^{-k}$ as a class in $\pi_{k(1-\lambda)}(\Sigma^{k\rho}\e(n-1))$. The multiplication by $y^{-k}$ map
$$\Sigma^{k(1-\lambda)}\e(n-1) \longrightarrow \Sigma^{k\rho}\e(n-1)$$
is an equivalence since $y$ is invertible. Applying this to each wedge summand (where we use different powers of $y$ for each suspension), we have an equivalence
$$\limt{i} \bigvee_{j \leq i} \Sigma^{k(1-\lambda)}\e(n-1)\simeq  \limt{i} \bigvee_{j \leq i} \Sigma^{\rho j} E\r (n-1).$$
Putting this together with the map of Theorem \ref{C2eqvtERsplit} and the equivalence with the classical Tate construction of Theorem \ref{Thm:Agree} and taking fixed points (applying Proposition \ref{Tatefixed}), we have a map
$$\limt{i} \bigvee_{j \leq i} \Sigma^{(1-\lambda) j} ER(n-1) \longrightarrow (ER(n)^{t\Sigma_2})^\wedge_{\widehat{I}_{n-1}}$$
which yields an isomorphism on homotopy groups after completion as desired. As in the parametrized case, the corresponding statement for $K(n-1)$ localization (or equivalently $KR(n-1)$-localization) follows from Theorem \ref{prop:localizationcomparison}.

We now prove the case $n=1$. We claim that
$$\left( \limt{i} \bigvee_{j \leq i} \Sigma^{\rho j} \e(0) \right)^{C_2} \simeq \limt{i} \bigvee_{j \leq i} \Sigma^{4j} H\q.$$
By \cite[Sec. 4]{KW07}, there is an isomorphism 
$$\pi^{C_2}_\star(\e(0)) \cong \z_{(2)}[2^{\pm 1}, u_{2\sigma}^{\pm 1}] \cong \q[u_{2\sigma}^{\pm 1}]$$
with $|u_{2\sigma}| = 2-2\sigma$. Therefore in integer-grading we have
$$\pi_*^{C_2}(\Sigma^{\rho j} \e(0)) \cong \begin{cases}
\q \quad & \text{ if } j=2k \text{ and } * = -4k, \\
0 \quad & \text{ if } \text{otherwise}. 
\end{cases}
$$
It follows that $(\Sigma^{(2k+1)\rho }\e(0))^{C_2} \simeq \ast$ and $(\Sigma^{2k\rho}\e(0))^{C_2} \simeq  \Sigma^{-4k}H\mathbb{Q}$. This proves the claim.
\end{proof}

\section{Tate vanishing}\label{Sec:TateVanishing}

Our second main result, vanishing for the parametrized Tate construction of Real Morava K-theory, is proven in this section. Let $G$ be any finite group and let $G$ act trivially on the $n$th Morava K-theory $K(n)$. Greenlees and Sadofsky proved the following vanishing result:

\begin{thm}\cite[Thm. 1.1]{GS96} 
The Tate construction of Morava K-theory vanishes, i.e. $K(n)^{tG} \simeq *$. 
\end{thm}

Their proof proceeds in two steps:
\begin{enumerate}
\item They show that the result holds for cyclic groups using the inverse limit formula for the Tate construction. Note that this inverse limit formula does not hold for the Tate construction of a general group.
\item To prove vanishing for general groups, they apply an inductive argument. The key point is that if the Tate constructions for the proper subgroups of $G$ vanish, then there is an understandable model for the Tate construction \cite[Prop. 3.2]{GS96}.
\end{enumerate}

Our proof for vanishing of the parametrized Tate construction follows a very similar line of reasoning.  In Section \ref{Sec:TVCyclic}, we show that $\k(n)^{t_{C_2}G} \simeq *$ for all finite cyclic groups $G$. In Section \ref{Sec:TVAll}, we extend this vanishing result to all finite groups. Along the way, we  prove finite generation for the Real Morava K-theory of $C_2$-equivariant classifying spaces; this result may be of independent interest.

\subsection{Parametrized Tate vanishing for cyclic groups}\label{Sec:TVCyclic}

Throughout this section, we assume that $G$ is a finite group. 

\begin{lem}\label{Lem:2.1}
Let $\xi$ be a positive dimensional Real vector bundle over $B_{C_2}G$ and let $\k$ be a Real oriented cohomology theory with $\k_\star(B_{C_2}G)$ finitely generated over $\e_\star$ (where $\mathbb{K}$ is an $\mathbb{E}$-module spectrum). Then
$$\limt{s} (\k \wedge B_{C_2}G^{(-s\xi)}) \simeq *.$$
\end{lem}

\begin{proof}
By \cite[Lem. 2.1]{GS96}, the composites of the underlying maps of the structure maps in the inverse limit are eventually null for dimension reasons. Therefore it suffices to show that the geometric fixed points of the composites of the structure maps in the inverse limit are eventually null. We do so by following Greenlees and Sadofsky's proof of \cite[Lem. 2.1]{GS96}. 

The assumption that $\k_\star(B_{C_2}G)$ is finitely generated over $\e_\star$ implies that each generator of $\k_\star(B_{C_2}G)$ is supported on some finite skeleton of $B_{C_2}G$, say $B_{C_2}G^{\langle r \rangle}$. Recall that the geometric fixed points of a Thom spectrum may be computed by taking the fixed points of each (suspension of a) Thom space in the spectrum \cite{Man04}. More precisely, we have
$$\Phi^{C_2}(Th(-s\xi \to B_{C_2}G^{\langle r \rangle})) \simeq Th((-s\xi|_{(B_{C_2}G^{\langle r \rangle})^{C_2}})^{C_2} \to (B_{C_2}G^{\langle r \rangle})^{C_2}).$$
Let $\lambda$ be the complex dimension of $-s\xi$. Then $(-s\xi)^{C_2}$ has real dimension $\lambda$. On the other hand, the real dimension $d$ of $(B_{C_2} G)^{C_2}$ satisfies $0 \leq d \leq r$. Indeed, this follows from the facts that $B_{C_2}G$ admits a $C_2$-CW structure for any discrete group $G$ \cite{Luc05} and the $C_2$-fixed points of an $r$-dimensional $C_2$-CW complex consist of the cells of the form $C_2/C_2 \times D^k$. In particular, the dimension does not increase after taking fixed points. Given these constraints on dimension, a staightforward modification of the dimension argument appearing in the proof of \cite[Lem. 2.1]{GS96} shows that the geometric fixed points of the composite
$$\k \wedge (B_{C_2}G^{\langle r \rangle})^{(-(s+j)\xi)} \to \k \wedge B_{C_2}G^{(-(s+j)\xi)} \to \k \wedge B_{C_2}G^{(-s\xi)}$$
are null for $j$ sufficiently large.
\end{proof}

\begin{rem2}
It may be possible to obtain tighter dimension bounds using concrete geometric models for $B_{C_2}G$, but we only needed rough bounds to apply the argument from \cite{GS96}. 
\end{rem2}

\begin{lem}\label{Lem:2.2}
Let $V$ be some finite, positive dimensional Real representation of $G$. Let $\k$ be a Real oriented cohomology theory such that $\k_\star(B_{C_2}H)$ is finitely generated over $\k_\star$ for all $H \leq G$. Then the $G \rtimes C_2$-spectrum
$$F(S^{\infty V}, i_*\k \wedge E_{C_2}G_+)$$
is equivariantly contractible. 
\end{lem}

\begin{proof}
We modify the proof of \cite[Lem. 2.2]{GS96}. For any spectrum $K$, 
$$F(S^{\infty V}, i_*K \wedge E_{C_2}G_+) \simeq \limt{r} F(S^{rV},i_*K \wedge E_{C_2}G_+) \simeq \limt{r} i_*K \wedge E_{C_2}G_+ \wedge S^{-rV}.$$
If $H \leq G$, then $V$ is also a Real representation of $H$, and $E_{C_2}G$ is a model for $E_{C_2}H$. So
\begin{equation}\label{Eqn:Equiv8}
(i_*\k \wedge E_{C_2}G_+ \wedge S^{-rV})^H \simeq \k \wedge B_{C_2}H^{(-r\xi)}
\end{equation}
where $\xi$ is the Real bundle over $B_{C_2}H$ induced by the Real $H$-representation $V$ and $B_{C_2}H^{(-r\xi)}$ is the $C_2$-equivariant Thom spectrum. The equivalence \ref{Eqn:Equiv8} follows from the proof of \cite[Thm. B]{QS21a}. We see then that
$$F(S^{\infty V}, i_*\k \wedge E_{C_2}G_+)^H \simeq \limt{r} \k \wedge B_{C_2}H^{(-r\xi)}.$$
The right-hand side is contractible for all $H \leq G$ by Lemma \ref{Lem:2.1}, so the left-hand side is equivariantly contractible. 
\end{proof}

\begin{lem}
The group $\k(n)^\star(B_{C_2}G)$ is finitely generated over $\e(n)^\star$. Dually, $\k(n)_\star(B_{C_2}G)$ is finitely generated over $\e(n)_\star$. 
\end{lem}

\begin{proof}


First suppose $n\geq 1$. Since both $\k(n)$ and $\e(n)$ are $(\lambda+\sigma)$-periodic, it suffices to show the statement for the integer-graded part. For any $\mathbb{MR}(n)$-module spectrum $\mathbb{E}$, we have a Kitchloo--Wilson fibration (Section \ref{sec:kwfibration})
$$\xymatrix{\Sigma^{\lambda}\mathbb{E}^{C_2} \ar@{->}[r]^-x &  \mathbb{E}^{C_2} \ar@{->}[r] & E.}$$
We plug in the $\mathbb{MR}(n)$-module spectrum $\mathbb{E}:=\k(n) \wedge B_{C_2}G$. Nonequivariantly, we have that $B_{C_2}G$ and $\k(n)$ are equivalent to $BG$ and $K(n)$, respectively. Thus, the underlying nonequivariant spectrum of $\k(n) \wedge B_{C_2}G$ is $K(n) \wedge BG$. We therefore have a fibration of spectra
$$\xymatrix{\Sigma^{\lambda}(\k(n) \wedge B_{C_2}G)^{C_2} \ar@{->}[r]^{x} &  (\k(n) \wedge B_{C_2}G)^{C_2}  \ar@{->}[r] & K(n) \wedge BG}.$$
Applying homotopy groups gives a long exact sequence of modules over $E(n)_*$. Note that $E(n)_*$ is finitely generated (with two generators) over $ER(n)_*$, so it suffices to show that $\pi_*(\k(n) \wedge B_{C_2}G)^{C_2}=\k(n)_*(B_{C_2}G)$ is finitely generated over $E(n)_*$.

Denote the kernel (resp. cokernel) of $x_* \colon \k(n)_*(B_{C_2}G) \overset{x_*}{\to} \k(n)_*(B_{C_2}G)$ by $\ker(x_*)$ (resp. $\coker(x_*)$). Then we have (we omit the grading shifts)
$$0 \rightarrow \ker(x_*) \rightarrow K(n)_*(BG) \rightarrow \coker(x_*) \rightarrow 0.$$
Because $K(n)_*(BG)$ is finitely generated over $K(n)_*$ \cite{Rav82}, it is finitely generated over $E(n)_*$. Since $E(n)_*$ is Noetherian, this implies that both $\ker(x_*)$ and $\coker(x_*)$ are finitely generated over $E(n)_*$. Note that
$\ker(x^2_*)/\ker(x_*)$ is a submodule of $\coker(x_*)$. Hence the short exact sequence of $E(n)_*$-modules
$$0 \rightarrow \ker(x_*) \rightarrow \ker(x^2_*) \rightarrow \ker(x^2_*)/\ker(x_*) \rightarrow 0$$
implies that $\ker(x^2_*)$ is finitely generated over $E(n)_*$. Inductively, we know $\ker(x^k_*)$ is finitely generated for all positive integers $k$. Because $x$ is nilpotent, $\k(n)_*(B_{C_2}G)$ equals $\ker(x^L_*)$ for $L$ large enough and it is finitely generated over $E(n)_*$. Thus, $\k(n)_*(B_{C_2}G)$ is finitely generated over $ER(n)_*$, and so $\k(n)_\star(B_{C_2}G)$ is finitely generated over $\e(n)_\star$, which completes the proof.

The proof of finite-generation of cohomology is completely analogous (one simply maps $B_{C_2}G$ into the Kitchloo--Wilson fibration rather than smashing with it).

In the case $n=0$, we use the (unshifted) fibration
$$\xymatrix{(\Sigma^{\sigma}B_{C_2}G \wedge \k(0))^{C_2} \ar@{->}[r]^-{a_{\sigma}} & (B_{C_2}G \wedge \k(0))^{C_2}  \ar@{->}[r] & BG \wedge K(0)}$$
and note that since mulitplication by $a_{\sigma}$ is null-homotopic on $\k(0)$, it follows that the homotopy of the middle term injects into the homotopy of the right hand term. Since $K(0)_*(BG)$ is finitely generated over $E(0)_*=\mathbb{Q}$, it follows that $\pi_*((B_{C_2}G \wedge \k(0))^{C_2})$ is as well.

\end{proof}

\begin{cor}\label{Cor:CycVanishing}
If $G$ is a finite cyclic group, then
$$(i_*\k(n))^{t_{C_2}G} \simeq *.$$
\end{cor}

\begin{proof}
The proof is similar to the proof of \cite[Cor. 2.3]{GS96}. Let $(V,\overline{(-)})$ be the Real representation of $G$ described in part $(2)$ of Example \ref{Exm:RealReps} and let $\xi$ be the corresponding Real line bundle over $B_{C_2}G$.  Then there are $G \rtimes C_2$-equivariant equivalences $E_{C_2}G_+ \simeq S(\infty V)_+$  and $\widetilde{E_{C_2}G} \simeq S^{\infty V}$. We have
$$(i_*\k(n))^{t_{C_2}G} \simeq F(S^{\infty V}, i_*\k(n) \wedge \Sigma S(\infty V)_+),$$
and the right-hand side is equivariantly contractible by Lemma \ref{Lem:2.2}. 
\end{proof}

\begin{rem2}
A similar proof can be given if $G$ is any finite abelian group, but we omit this since we will give a proof for all finite groups below. 
\end{rem2}

\subsection{Parametrized Tate vanishing for all finite groups}\label{Sec:TVAll}

Suppose that $G$ is a finite group. The proofs of the following two propositions are similar to the proofs of \cite[Prop. 3.1-3.2]{GS96}. 

\begin{prop}\label{Prop:3.1}
If $\k$ is a Real oriented cohomology theory with $\k_\star({B_{C_2}G}_+)$ finitely generated over $\k_\star$ for all finite groups $G$, then $(i_*\k)^{t_{C_2} G} \simeq *$ for all finite groups $G$. 
\end{prop}

\begin{prop}\label{Prop:3.2}
If $W$ is a non-zero, finite dimensional Real $G$-representation with $W^G = 0$, and $\mathbb{K}$ is a $G \rtimes C_2$-spectrum such that for every proper subgroup $H < G$ one has $\mathbb{K}^{t_{C_2}H} \simeq *$, then
$$F(S^{\infty W}, \Sigma \mathbb{K} \wedge E_{C_2} G_+) \simeq \mathbb{K}^{t_{C_2}G}.$$
\end{prop}

\begin{proof}[Proof of Prop. \ref{Prop:3.1}]
This follows from the chain of equivariant equivalences
$$(i_*\k)^{t_{C_2} G} \simeq F(S^{\infty W}, \Sigma i_* \k \wedge E_{C_2}G_+) \simeq *$$
where the first equivalence is Proposition \ref{Prop:3.2} and the second equivalence is Lemma \ref{Lem:2.2}. Proposition \ref{Prop:3.2} applies by induction up the subgroup lattice of $G$, where the base case of a cyclic subgroup was proven in Corollary \ref{Cor:CycVanishing}.
\end{proof}

\begin{proof}[Proof of Prop. \ref{Prop:3.2}]
By hypothesis, for any proper subgroup $H < G$, we have
$$\mathbb{K}^{t_{C_2}H} \simeq F(\widetilde{E_{C_2}G}, \Sigma \mathbb{K} \wedge E_{C_2}G_+)^{H} \simeq *$$
(note that $E_{C_2}G$ is a model for $E_{C_2} H$ by definition). We also have $F(S^{\infty W},\Sigma \mathbb{K} \wedge E_{C_2}G_+)^H \simeq *$ since if $W^H = 0$ we can apply Proposition \ref{Prop:3.2} inductively and if $W^H \neq 0$ then $S^{\infty W} \simeq *$. 

Therefore it suffices by the $G$-Whitehead theorem to produce a map
$$F(S^{\infty W},\Sigma \mathbb{K} \wedge E_{C_2}G_+) \to F(\widetilde{E_{C_2}G}, \Sigma \mathbb{K} \wedge E_{C_2}G_+)$$
which is an equivalence on $G$-fixed points. Smash $S^{\infty W}$ with the cofibration 
\begin{equation}\label{Eqn:9}
E_{C_2}G_+ \to S^0 \to \widetilde{E_{C_2}G}. 
\end{equation}
The $G \rtimes C_2$-spectrum $S^{\infty W} \wedge E_{C_2}G_+$ is equivariantly contractible. Therefore by smashing with \ref{Eqn:9} we have a $G \rtimes C_2$-equivalence
\begin{equation}\label{Eqn:10}
S^{\infty W} \to S^{\infty W} \wedge \widetilde{E_{C_2}G}.
\end{equation}

Now, we have $(S^{\infty W})^G \simeq S^0$ as spaces, so
$$(S^{\infty W}/S^0)^G \simeq *$$
and therefore $(S^{\infty W}/S^0)$ can be built from $G$-cells of the form $(G/H)_+ \wedge E^n$ where $H < G$ is a proper subgroup. Using equivalences from the proof of \cite[Thm. B]{QS21a} gives 
$$F(G/H_+ \wedge S^n \wedge \widetilde{E_{C_2}G}, \Sigma \mathbb{K}\wedge E_{C_2}G_+) \simeq F(S^n \wedge \widetilde{E_{C_2}G}, \Sigma \mathbb{K} \wedge E_{C_2} G_+)^H \simeq \Sigma^{-n} \mathbb{K}^{t_{C_2}H} \simeq *.$$
Taking the limit over skeleta of $S^{\infty W} /S^0$ gives
$$F((S^{\infty W}/S^0) \wedge \widetilde{E_{C_2}G}, \Sigma \mathbb{K} \wedge E_{C_2}G_+)^G \simeq *.$$
The desired map is then given by 
\begin{align*}
F(S^{\infty W}, \Sigma \mathbb{K} \wedge E_{C_2}G_+) & \simeq F(S^{\infty W} \wedge \widetilde{E_{C_2}G},\Sigma \mathbb{K} \wedge E_{C_2}G_+) \\
& \to F(S^0 \wedge \widetilde{E_{C_2}G},\Sigma \mathbb{K} \wedge E_{C_2}G_+) \simeq F(\widetilde{E_{C_2}G}, \Sigma \mathbb{K} \wedge E_{C_2}G_+).
\end{align*}
\end{proof}

Putting all of this together, we have the following:

\begin{thm}\label{Thm:TVMain}
Suppose $G$ is a finite abelian group and $\mathbb{K} = i_*\k(n)$. Then 
$$\mathbb{K}^{t_{C_2}G} \simeq *.$$
\end{thm}

\begin{defin}
We will say that a $C_2$-spectrum $K$ is an \emph{integral Real Morava K-theory} if $K$ is Real oriented, has no torsion in its homotopy groups, and reduces to $\k(n)$ modulo $p$. 
\end{defin}

\begin{cor}
If $E = i_*K$, where $K$ is a $p$-local integral Real Morava K-theory, then $E^{t_{C_2}G}$ is rational.
\end{cor}

\begin{proof}
The proof is similar to the proof of \cite[Cor. 1.4]{GS96}. 
\end{proof}

\begin{rem2}
\begin{enumerate}
\item In \cite[Sec. 4]{GS96}, Greenlees and Sadofsky observed that the same proof holds for integral theories. This is true in the Real oriented case as well; for example, we see that $K\r^{t_{C_2}G}$ is rational. 
\item Greenlees and Sadofsky also showed that an analogous statement holds for $KO$ using the Wood cofiber sequence
$$\Sigma KO \overset{\eta}{\to} KO \to KU$$
and the fact that $\eta$ is nilpotent classically. Although there is a $C_2$-equivariant analog of the Wood cofiber sequence (involving the spectrum $KO_{C_2}$, see e.g. \cite[Sec. 10]{GHIR19}), the same argument does not work since $\eta$ is not nilpotent $C_2$-equivariantly. 
\end{enumerate}
\end{rem2}

We will now prove the real (i.e. fixed point) version of the Tate vanishing result.

\begin{thm}\label{Thm:realvanishing}
For a finite cyclic group $G$, $K\mathbb{R}(n)^{tG} \simeq *$ where $G$ acts trivially on $K\mathbb{R}(n)$. In particular, $KR(n)^{tG} =(K\mathbb{R}(n)^{tG})^{C_2}\simeq *$.
\end{thm}
\begin{proof}
In the case that $n=0$, we have $KR(0)=H\mathbb{Q}$ and Tate vanishing is the same as the nonequivariant result. Suppose $n \geq 1$. We will show that the $RO(C_2)$-graded homotopy groups of $K\mathbb{R}(n)^{tG}$ vanish. Because $K\mathbb{R}(n)$ is $\lambda+\sigma$-periodic, it is enough to compute the integer part $\pi_*^{C_2} K\mathbb{R}(n)^{tG}=0$. We will show that
$$\pi_*^{C_2} K\mathbb{R}(n)^{tG}=\pi_* (K\mathbb{R}(n)^{tG})^{C_2}=\pi_*(KR(n))^{tG}=0.$$
We write $(KR(n))^{tG}$ as $\limt{n} (KR(n)\wedge \Sigma RP^\infty_{-n})$. The cofiber sequence \cite{KW07}
$$\Sigma^\lambda KR(n) \xrightarrow{y} KR(n) \rightarrow K(n)$$
gives cofiber sequences (we omit suspensions)
$$\Sigma^{\lambda} KR(n)\wedge BG^\infty_{-k} \xrightarrow{f_k} KR(n) \wedge BG^\infty_{-k} \rightarrow K(n) \wedge BG^\infty_{-k}.$$
Denote the cokernel of $(f_k)_* \colon \pi_{l-\lambda}(KR(n) \wedge BG^\infty_{-k}) \rightarrow \pi_l(KR(n) \wedge BG^\infty_{-k})$ by $A_{k,l}$, the kernel of $(f_k)_* \colon \pi_{l-\lambda-1}(KR(n) \wedge BG^\infty_{-k}) \rightarrow \pi_{l-1}(KR(n) \wedge BG^\infty_{-k})$ by $C_{k,l}$, and $\pi_l (K(n) \wedge BG^\infty_{-k})$ by $B_{k,l}$. The associated long exact sequence in homotopy groups
$$\cdots \pi_{l-\lambda}(KR(n) \wedge BG^\infty_{-k}) \xrightarrow{(f_k)_*} \pi_l(KR(n) \wedge BG^\infty_{-k}) \rightarrow \pi_l (K(n) \wedge BG^\infty_{-k}) \xrightarrow{\partial_l} \pi_{l-1-\lambda} (KR(n) \wedge BG^\infty_{-k}) \cdots$$
breaks into short exact sequences
$$0\rightarrow A_{k,l} \rightarrow B_{k,l} \rightarrow C_{k,l} \rightarrow 0.$$
We omit $l$ from the index. We have an exact sequence
\begin{equation}\label{equation:lim1}
0 \rightarrow \limt{k} A_k \rightarrow \limt{k} B_k \rightarrow \limt{k} C_k \rightarrow \text{lim}^1A_k \rightarrow \text{lim}^1B_k \rightarrow \text{lim}^1C_k \rightarrow  0
\end{equation}
On the other hand, we have
\begin{equation}\label{equation:lim12}
0 \rightarrow \text{lim}^1 B_k \rightarrow \pi_*(K(n)^{tG}) \rightarrow \limt{k} B_k \rightarrow 0.
\end{equation}
Classical Tate vanishing \cite[Prop. 3.1]{GS96} implies that in the short exact sequence \ref{equation:lim12} we have  $\text{lim}^1 B_k=\limt{k} B_k = 0$. In the exact sequence \ref{equation:lim1}, we have
\[
\limt{k} A_k= \text{lim}^1 C_k = 0.
\]
The maps
\[
(f_k)_* \colon \pi_{l-\lambda}(KR(n) \wedge BG^\infty_{-k}) \rightarrow \pi_l(KR(n) \wedge BG^\infty_{-k})
\]
induces maps
\[
\lim f_* \colon \limt{k} \pi_{l-\lambda}(KR(n) \wedge BG^\infty_{-k}) \rightarrow \limt{k} \pi_l(KR(n) \wedge BG^\infty_{-k});
\]
\[
\text{lim}^1 f_* \colon \text{lim}^1_k \, \pi_{l-\lambda}(KR(n) \wedge BG^\infty_{-k}) \rightarrow \text{lim}^1_k \, \pi_l(KR(n) \wedge BG^\infty_{-k}).
\]
Denote the map $KR(n)^{tG} \rightarrow KR(n)^{tG}$ from $f_k$ by $f$. Then $f$ is nilpotent because it is induced by multiplication by an element of $\pi_* \mathbb{MR}(n)$. 

Recall that $A_k$ is the cokernel of $(f_k)_* \colon \pi_{l-\lambda}(KR(n) \wedge BG^\infty_{-k}) \rightarrow \pi_l(KR(n) \wedge BG^\infty_{-k})$ and $C_k$ is the kernel of $(f_k)_* \colon \pi_{l-\lambda-1}(KR(n) \wedge BG^\infty_{-k}) \rightarrow \pi_{l-1}(KR(n) \wedge BG^\infty_{-k})$. The short exact sequence
\[
0 \rightarrow \text{lim}^1 \pi_{*+1}(KR(n) \wedge BG^\infty_{-k}) \rightarrow \pi_* (KR(n)^{tG} \rightarrow  \limt{k} \pi_* (KR(n) \wedge BG^\infty_{-k}) \rightarrow 0
\]
has a self map $( \text{lim}^1 (f_k)_*, f_*, \lim (f_k)_*)$. Because $f_*$ is nilpotent, both $\text{lim}^1 (f_k)_*$ and $\lim (f_k)_*$ are nilpotent.

The exact sequence
$$0 \rightarrow C_k \rightarrow \pi_*(KR(n) \wedge BG^\infty_{-k}) \xrightarrow{(f_k)_*}  \pi_*(KR(n) \wedge BG^\infty_{-k}) \rightarrow A_k \rightarrow 0$$
splits into two short exact sequences
$$0 \rightarrow C_k \rightarrow \pi_*(KR(n) \wedge BG^\infty_{-k}) \xrightarrow{(f_k)_*}  \text{im}(f_k)_* \rightarrow 0,$$
$$0 \rightarrow \text{im}(f_k)_*  \rightarrow  \pi_*(KR(n) \wedge BG^\infty_{-k}) \rightarrow A_k \rightarrow 0.$$
We then have exact sequences
\[
\begin{tikzcd}[row sep=small]
0 \arrow{r} &  \limt{k} C_k \arrow{r} &  \limt{k} \pi_*(KR(n) \wedge BG^\infty_{-k}) \arrow{r} &  \limt{k} \text{im}(f_k)_* \arrow{lld} \\
		&  \text{lim}^1 C_k  \arrow{r} &  \text{lim}^1 \pi_*(KR(n) \wedge BG^\infty_{-k}) \arrow{r} & \text{lim}^1 \text{im}(f_k)_* \arrow{r} & 0,
\end{tikzcd}
\]
\[
\begin{tikzcd}[row sep=small]
0 \arrow{r} & \limt{k} \text{im}(f_k)_*  \arrow{r} & \limt{k} \pi_*(KR(n) \wedge BG^\infty_{-k}) \arrow{r} & \limt{k} A_k  \arrow{lld} \\
		&  \text{lim}^1 \text{im}(f_k)_*  \arrow{r} &  \text{lim}^1 \pi_*(KR(n) \wedge BG^\infty_{-k}) \arrow{r} &  \text{lim}^1 A_k \arrow{r} & 0.
\end{tikzcd}
\]

Because $\text{lim}^1 C_k=0$, the map $\text{lim}^1 \pi_*(KR(n) \wedge BG^\infty_{-k}) \rightarrow \text{lim}^1 \text{im}(f_k)_*$ is injective (in fact, it is an isomorphism). Because $\limt{k} A_k=0$, the map $\text{lim}^1 \text{im}(f_k)_*  \rightarrow  \text{lim}^1 \pi_*(KR(n) \wedge BG^\infty_{-k})$ is injective. Therefore, the composition
$$\text{lim}^1 (f_k)_* \colon \text{lim}^1 \pi_*(KR(n) \wedge BG^\infty_{-k}) \rightarrow \text{lim}^1 \pi_*(KR(n) \wedge BG^\infty_{-k})$$
is injective. We have shown $\text{lim}^1 (f_k)_*$ is nilpotent and so $\text{lim}^1 \pi_*(KR(n) \wedge BG^\infty_{-k})=0$. With a similar argument, $\limt{k} (f_k)_*$ is surjective and nilpotent, which forces $\limt{k} \pi_*(KR(n) \wedge BG^\infty_{-k})=0$. Then $\pi_* KR(n)^{tG}=0$ and this completes the proof.
\end{proof}

We may now apply \cite[Props. 3.1 and 3.2]{GS96} to obtain Tate vanishing for general finite groups. 

\begin{cor}
For a finite group $G$, $K\r(n)^{tG} \simeq *$ where $G$ acts trivially on $K\r(n)$. In particular, $KR(n)^{tG} \simeq (K\r(n)^{tG})^{C_2} \simeq *.$
\end{cor}

\bibliographystyle{alpha}
\bibliography{master}

\end{document}